\documentclass{amsart}

\usepackage{amssymb,amsmath,amsfonts}
\usepackage{a4wide}
\usepackage{pstricks}

\newcommand{\f}{\ensuremath{\varphi}}
\newcommand{\p}{\ensuremath{\psi}}
\newcommand{\x}{\ensuremath{\chi}}

\newcommand{\si}{\sigma}
\newcommand{\De}{\Delta}
\newcommand{\Ga}{\Gamma}
\newcommand{\Si}{\Sigma}

\newcommand{\mdl}[1]{\models_{\lgc{#1}}}

\newcommand{\imp}{\Rightarrow}

\newcommand{\lang}{{\mathcal L}}
\newcommand{\cls}[1]{\mathcal{#1}}
\newcommand{\alg}[1]{\mathbf{#1}}
\newcommand{\spa}[1]{
  \smash{%
    \raise-12pt\hbox{%
      \vbox{\hsize  7.5pt%
        \hbox{\centerline{$#1$}}      	
        \hbox{\raise9pt\hbox{\,{\tiny $\sim$}}}
      }%
    }%
  }
}

\newcommand{\cop}[1]{\mathbb{#1}}
\newcommand{\eq}{\approx}
\newcommand{\leqn}{\preccurlyeq}

\newcommand{\lgc}[1]{\mathrm{#1}}

\newcommand{\id}{\mathrm{id}}
\newcommand{\Ug}{{\ensuremath{\mathbb{U}}}}
\newcommand{\Ugp}{{\ensuremath{\mathbb{U}^+}}}
\newcommand{\Qg}{{\ensuremath{\mathbb{Q}}}}
\newcommand{\Vg}{{\ensuremath{\mathbb{V}}}}
\newcommand{\K}{{\ensuremath{\mathcal{K}}}}
\newcommand{\F}{{\ensuremath{\alg{F}}}}
\newcommand{\Q}{{\ensuremath{\mathcal{Q}}}}

\newcommand{\U}{{\ensuremath{\mathcal{U}}}}

\DeclareMathOperator{\DM}{\mathcal{DMA}}
\DeclareMathOperator{\KA}{\mathcal{KA}}
\DeclareMathOperator{\BDL}{\mathcal{BDL}}
\DeclareMathOperator{\ST}{\mathcal{ST}}

\DeclareMathOperator{\XDM}{\mathsf{D_{\DM}}}
\DeclareMathOperator{\ADM}{\mathsf{E_{\DM}}}
\DeclareMathOperator{\XK}{\mathsf{D_{\KA}}}
\DeclareMathOperator{\AK}{\mathsf{E_{\KA}}}
\DeclareMathOperator{\XBDL}{\mathsf{D_{\BDL}}}
\DeclareMathOperator{\ABDL}{\mathsf{E_{\BDL}}}
\DeclareMathOperator{\XST}{\mathsf{D_{\ST}}}
\DeclareMathOperator{\AST}{\mathsf{E_{\ST}}}
\DeclareMathOperator{\XA}{\mathsf{D_{\cls{A}}}}
\DeclareMathOperator{\AAA}{\mathsf{E_{\cls{A}}}}
\DeclareMathOperator{\XB}{\mathsf{D_{\cls{B}}}}
\DeclareMathOperator{\AAB}{\mathsf{E_{\cls{B}}}}

\newtheorem{theorem}{Theorem}

\newtheorem{lemma}[theorem]{Lemma}
\newtheorem{example}[theorem]{Example}
\newtheorem{corollary}[theorem]{Corollary}
\newtheorem{proposition}[theorem]{Proposition}

\begin{document}


\title{Admissibility via Natural Dualities}

\author{Leonardo Cabrer and George Metcalfe}

\address{Mathematics Institute, University of Bern\\
Sidlerstrasse 5, 3012 Bern, Switzerland}
\email{lmcabrer@yahoo.com.ar}
\email{george.metcalfe@math.unibe.ch}


\begin{abstract}
It is shown that admissible clauses and quasi-identities of quasivarieties generated by 
a single finite algebra, or equivalently, the quasiequational and universal theories of their 
free algebras on countably infinitely many generators, may be characterized using natural dualities. 
In particular, axiomatizations are obtained for the admissible clauses and quasi-identities of 
bounded distributive lattices, Stone algebras, Kleene algebras and lattices, 
and De~Morgan algebras and lattices.
\end{abstract}


\maketitle


\section{Introduction}

The concept of admissibility was  introduced explicitly by Lorenzen in the 1950s in the 
context of intuitionistic logic~\cite{Lor55}, but played an important role already in the earlier work of 
Gentzen~\cite{Gen35}, Johansson~\cite{Joh36}, and Whitman~\cite{Whi41}. Informally, 
a rule, consisting of a finite set of premises and a conclusion, is admissible for a system 
if adding it to the system produces no new theorems. Equivalently, the rule is 
admissible if for each of its instances, whenever the premises are theorems, the 
conclusion is also a theorem. Typically, proving that a rule is admissible 
for a system establishes that the system, or corresponding  class of structures, 
has a certain property. For example, proving the admissibility of a ``cut'' rule for a proof system 
is a standard proof-theoretic strategy for establishing the completeness of the system for a given semantics 
(see, e.g.~\cite{Gen35,MOG08}). Similarly, proving the admissibility of a ``density'' rule for a proof 
system has been used to establish that certain classes of algebraic structures are generated as 
varieties by their dense linearly ordered members (see, e.g.~\cite{MM07,MOG08}). More generally, 
the addition of admissible rules to a system may be used to shorten proofs or speed up proof search. 

In classical propositional logic, admissibility and derivability coincide; that is, the logic is structurally complete. 
However, this is  not the case for non-classical logics such as intuitionistic logic and certain modal, many-valued, and substructural logics 
which formalize reasoning about, e.g., necessity and knowledge, vagueness and probability, and bounded resources and relevance. 
In particular, Rybakov proved in~\cite{Ryb84} (see also the monograph~\cite{Ryb97}) that the set of admissible rules of 
intuitionistic propositional logic is decidable but not finitely axiomatizable. 
Axiomatizations (called bases) of the admissible rules for this logic were later provided,  
independently, by Iemhoff~\cite{Iem01} and Rozi{\`e}re~\cite{Roz92}, and subsequently 
for other intermediate logics by Iemhoff~\cite{Iem05} and Cintula and Metcalfe~\cite{CM10},  transitive modal logics 
by Je{\v r}{\'a}bek~\cite{Jer05}, temporal logics by Babyneshev and Rybakov~\cite{BR11a,BR11b}, 
and \L ukasiewicz many-valued logics  by Je{\v r}{\'a}bek~\cite{Jer09a,Jer09b}. 
Bases are also given in these works for admissible multiple-conclusion rules where admissibility means that for each 
instance of the rule, whenever the premises are theorems, one of the conclusions is a theorem.

An algebraizable logic -- roughly speaking, a logic with well-behaved translations between formulas and 
identities -- is associated with a class of algebras called the equivalent algebraic semantics of the logic. 
This class is usually a quasivariety, in which case, rules correspond to quasi-identities and 
multiple-conclusion rules correspond to clauses (implications between conjunctions of identities and disjunctions of identities). 
Admissibility then amounts to the validity of quasi-identities or clauses in the 
free algebra on countably infinitely many generators of the quasivariety. Hence the study of 
admissibility may be viewed algebraically as the study of quasiequational and universal theories of 
free algebras. In particular, a crucial role is played here by subalgebras of (powers of) the free algebra 
on countably infinitely many generators. Indeed, in certain cases these algebras may be described 
and investigated using dualities. Notably, characterizations of admissible rules for intermediate 
and modal logics have  been obtained using Kripke 
frames~\cite{Ghi99,Ghi00,Iem01,Iem05,Jer05,CM10}, and for infinite-valued 
\L ukasiewicz logic using rational polyhedra~\cite{Jer09a,Jer09b}. 

For quasivarieties generated by a single finite algebra (e.g., Boolean algebras, distributive lattices, Stone algebras, 
Kleene algebras, De~Morgan algebras), an accommodating and powerful 
dual framework may be provided by the theory of natural dualities (see~\cite{CD98}). 
In this setting, each algebra in the quasivariety is represented as an algebra of continuous 
morphisms over some structured Stone space, and there exists a dual equivalence 
between the quasivariety and the corresponding category of structured Stone spaces.
Crucially for the approach described here, in the presence of a natural duality, 
free algebras of the quasivariety have a convenient representation on the dual side. 
The main goal of this paper is to show how natural dualities can be used to obtain bases for the  
admissible clauses and quasi-identities of quasivarieties generated by a finite algebra, or equivalently, axiomatizations 
of the quasiequational and universal theories of their free algebras on countably infinitely many generators. 
More precisely, it is shown that these bases can be derived from characterizations of the finite subalgebras 
of (finite powers of) the free algebras on countably infinitely many generators. These 
characterizations are obtained in turn via a combinatorial description of the dual objects of the algebras. 
Note that algorithms for checking admissibility of quasi-identities (but not clauses) in finitely generated quasivarieties 
have been obtained algebraically in~\cite{MR12,MR13} but general methods for characterizing and axiomatizing 
admissible quasi-identities and clauses has until now been lacking.

As case studies, bases are obtained for the admissible quasi-identities and clauses of 
bounded distributive lattices, Stone algebras, De~Morgan algebras and lattices, and Kleene algebras and lattices. 
Not only do these classes of algebras possess natural dualities already available in the literature~\cite{CD98}, 
they have also been proposed as suitable semantics for applications in computer science.  For example, De~Morgan algebras 
(also called distributive $i$-lattices or quasi-Boolean algebras) have been used as suitable models for reasoning in databases 
(containing, e.g., extra elements representing ``both true and false'' and ``neither true nor false'') and as the basis for fuzzy logic systems 
(see, e.g., \cite{Bel77,Geh2003}). 
Some of these classes of algebras can be tackled 
without natural dualities (in particular, admissibility in bounded distributive lattices 
and Stone algebras admits a straightforward algebraic treatment). However, for other cases, 
such as Kleene and De Morgan algebras, 
the natural dualities play a crucial role in rendering the task feasible. A central goal of this paper  
is to emphasize the uniform approach to tackling these problems. 
 
 Following some preliminaries  in Section~\ref{s:preliminaries} 
 on universal algebra and natural dualities, Section~\ref{s:admissibility} develops the necessary 
 tools for characterizing admissible clauses and quasi-identities, 
 focussing on locally finite quasivarieties. This approach is 
 illustrated first in Section~\ref{s:distributive} with the relatively straightforward cases of (bounded) 
 distributive lattices and Stone algebras. Sections~\ref{s:demorgan} and~\ref{s:kleene} then treat, 
 respectively, the more technically demanding cases of De~Morgan algebras and lattices and Kleene algebras and lattices.


\section{Preliminaries}\label{s:preliminaries}

Let us begin by recalling some basic material on universal algebra and natural dualities that will be needed in later sections. 
The reader is referred to~\cite{ML71} for background and undefined notions in category theory and to~\cite{BS81} and~\cite{CD98} for 
detailed explanations, definitions, and proofs concerning universal algebra and natural dualities, respectively.


\subsection{Universal Algebra}

For convenience, let us assume in subsequent discussions that $\lang$ is a finite  
algebraic language and that an {\em $\lang$-algebra}  $\alg{A}$ is an algebraic structure 
for this language, called {\em trivial} if $|A| = 1$.  
We denote the formula algebra (absolutely free algebra) for $\lang$ over countably infinitely many variables (free generators) 
by $\alg{Fm_\lang}$ and let the metavariables  $\f,\p,\x$ stand for  members of ${\rm Fm}_\lang$ 
called {\em $\lang$-formulas}.  An {\em $\lang$-identity} is an ordered pair of $\lang$-formulas, written $\f \eq \p$, and 
we let the metavariables $\Si, \De$  stand for  finite sets of  $\lang$-identities. We 
also use $\f \leqn \p$ as an abbreviation for $\f \wedge \p \eq \f$ in the presence of the lattice meet operation. 

An {\em $\lang$-clause} is an ordered pair of finite sets of $\lang$-identities, written $\Si \imp \De$. 
An $\lang$-clause $\Si \imp \De$ is   an {\em $\lang$-quasi-identity} if $|\De| = 1$, an
 {\em $\lang$-positive clause} if $\Si = \emptyset$, and an  {\em $\lang$-negative clause} if $\De = \emptyset$. 
 If both $|\De| = 1$ and $\Si = \emptyset$, then the $\lang$-clause is identified with the single $\lang$-identity in $\De$.
  
Now let $\K$ be a class of $\lang$-algebras  and let $\Si \imp \De$ be an  $\lang$-clause. We write 
$\Si \mdl{\K} \De$ to denote that for every $\alg{A} \in \K$ and homomorphism $h \colon \alg{Fm_\lang} \to \alg{A}$, 
whenever $\Si \subseteq \ker h$, then $\De \cap \ker h \not = \emptyset$. 
In this case, we say that each $\alg{A} \in \K$ {\em satisfies} 
$\Si \imp \De$ and that $\Si \imp \De$ is {\em valid} in $\K$. 
That is,  $\Si \imp \De$ may be understood as the universal formula 
$(\forall \bar{x}) (\bigwedge \Ga \imp \bigvee \De)$ where $\bar{x}$ are the variables occurring in $\Ga \cup \De$, and 
$\bigwedge \emptyset$  and $\bigvee \emptyset$ are identified with the truth constants $1$ and $0$, respectively, of first-order logic. 
Conversely, an arbitrary universal formula of the language 
$\lang$ may be associated with a finite set of $\lang$-clauses, by putting the quantified 
formula into conjunctive normal form and interpreting each conjunct as an $\lang$-clause.

We abbreviate $\emptyset \mdl{\K} \De$ by $\mdl{\K} \De$, and $\Si \mdl{\{\alg{A}\}} \De$ by $\Si \mdl{\alg{A}} \De$.  
We also drop the brackets in $\Si,\De$ when no confusion may occur. As usual, if the language is clear from the context 
we may omit the prefix $\lang$ when referring to $\lang$-algebras, $\lang$-clauses, etc.
 
A set of $\lang$-clauses $\Lambda$ is said to {\em axiomatize} $\K$ when $\alg{A} \in \K$ 
if and only if (henceforth ``iff'') all the clauses in $\Lambda$ are satisfied by $\alg{A}$.  More generally, given another 
class $\K'$ of $\lang$-algebras, $\Lambda$ is said to {\em axiomatize $\K$ relative to $\K'$} when 
$\alg{A} \in \K$ iff both $\alg{A} \in \K'$ and $\alg{A}$ satisfies all the clauses  in $\Lambda$.
 
$\K$ is called an $\lang$-{\em universal class} if there exists a set of $\lang$-clauses $\Lambda$ that axiomatizes $\K$. 
If there exists such a $\Lambda$ consisting only of $\lang$-identities, $\lang$-quasi-identities, or $\lang$-positive clauses, 
then $\K$ is called, respectively, an $\lang$-{\em variety}, $\lang$-{\em quasivariety}, or $\lang$-{\em positive universal class}.  
The \emph{variety $\Vg(\K)$}, \emph{quasivariety $\Qg(\K)$}, \emph{positive universal class} $\Ugp(\K)$, 
and \emph{universal class $\Ug(\K)$ generated by} $\K$ are, respectively, the smallest 
$\lang$-variety, $\lang$-quasivariety, $\lang$-positive universal class, and $\lang$-universal class  containing $\K$. 

Let $\cop{H}$, $\cop{I}$, $\cop{S}$, $\cop{P}$, and $\cop{P}_U$ be, respectively, the class operators of 
taking homomorphic images, isomorphic images, subalgebras, products, and ultraproducts.  Then
\[
\begin{array}{rclcrcl}
\Vg(\K) & = & \cop{HSP}(\K) 		& \qquad	& \Qg(\K) 		& = & \cop{ISPP}_U(\K)\\[.05in]
\Ugp(\K) & = & \cop{HSP}_U(\K) 	& \qquad 	&  \Ug(\K) & = & \cop{ISP}_U(\K).
\end{array}
\]
 Moreover, if $\K$ is a finite set of finite algebras, then 
 \[
\begin{array}{rclcrclcrcl}
\Qg(\K) & = & \cop{ISP}(\K) & \qquad & \Ugp(\K) & = & \cop{HS}(\K)  & \qquad & \Ug(\K) & = & \cop{IS}(\K).
\end{array}
\]
We refer to~\cite[Theorems II.9.5, II.11.9, V.2.20, and V.2.25]{BS81} and~\cite[Exercise 3.2.2]{CK77} for further details. 

Let $\K$ be a class of $\lang$-algebras and $\kappa$ a cardinal. An algebra $\alg{B} \in \K$ is called a 
\emph{free $\kappa$-generated algebra for $\K$}  if there exists a set $X$ and a map $g\colon X\to \alg{B}$ such that 
$|X|=\kappa$, $g[X]$ generates $\alg{B}$, and for every 
$\alg{A} \in \K$ and map $f \colon X \to \alg{A}$ there exists a (unique)
homomorphism $h \colon \alg{B} \to \alg{A}$ satisfying $f=h\circ g$. In this case, each $x \in X$ 
is called a {\em free generator} of $\alg{B}$.

For any quasivariety $\Q$, there exists for each cardinal $\kappa$, a unique (modulo isomorphism) 
free $\kappa$-generated algebra $\F_{\Q}(\kappa)$ in $\Q$. If $\Q$ contains at least one non-trivial algebra, then the map 
$g$ above is one-to-one and we may identify $X$ with the set $g[X]$. 
 We denote the special case of the free $\omega$-generated 
algebra of $\Q$ by $\F_{\Q}$. 

Let us recall some useful facts and properties. First, 
 $\F_{\Vg(\Q)}=\F_\Q$  (see~\cite[Corollary II.11.10]{BS81}) and hence for any algebra $\alg{A}$, it follows that 
 $\F_{\Vg(\alg{A})} = \F_{\Vg(\Qg(\alg{A}))} = \F_{\Qg(\alg{A})}  \in \Qg(\alg{A})$. Next note that 
identifying the free generators of $\alg{Fm}_\lang$ with the free generators of $\F_{\Q}$,  there exists a unique homomorphism 
$h_\Q\colon \alg{Fm_\lang} \to \F_\Q$ extending the identity map. 
We recall (see~\cite[Corollary II.11.6]{BS81}) that, for each 
$\lang$-identity $\f \eq \p$,
\begin{equation}\tag{I}\label{Eq:ValidityFree}
\mdl{\Q} \f \eq \p \qquad {\rm iff } \qquad \mdl{\F_\Q} \f \eq \p  \qquad {\rm iff } \qquad h_\Q(\f) = h_\Q(\p). 
\end{equation}
We note also that, since quasivarieties are closed under taking direct products, 
for any finite set of $\lang$-identities $\Si \cup \De$, 
\begin{equation}\tag{II}\label{Eq:ValidityQuasi}
\Si \mdl{\Q} \De \qquad {\rm iff } \qquad  \Si \mdl{\Q} \f \eq \p\ \textrm{ for some }\f \eq \p \in \De.
\end{equation}
Finally, it is helpful to observe that
\begin{equation}\tag{III}\label{Eq:ValidityPos}
 \mdl{\Q} \De \qquad {\rm iff } \qquad  \mdl{\Q} \f \eq \p\ \textrm{ for some }\f \eq \p \in \De  \qquad {\rm iff } \qquad \mdl{\F_\Q} \De.
\end{equation}
It only remains to check here that  $\mdl{\F_\Q} \De$ implies $\mdl{\Q} \De$ as the other directions follow from 
(\ref{Eq:ValidityQuasi}) and  $\F_\Q \in \Q$. Because $\De$ is finite, 
it involves only $n$ variables for some $n \in \mathbb{N}$, and hence $\mdl{\Q} \De$ iff 
 $\mdl{\alg{A}} \De$ for each 
$n$-generated algebra $\alg{A}$ in $\Q$. But each such algebra  is a homomorphic image of $\F_\Q$. 
So if $\mdl{\F_\Q} \De$, then, since homomorphisms preserve positive universal clauses, also $\mdl{\Q} \De$.


\subsection{Natural Dualities}

A {\em structured topological space} is a topological space $(X,\tau)$ together with a (possibly empty) finite set of relations 
defined on $X$ and a finite (possibly empty) set of internal operations on $X$. The general theory of natural dualities  
allows partial operations; however, the operations involved in the dualities considered in this paper will all be total.

Let us fix for the rest of this section a finite algebra $\alg{A}$ and  a structured topological
 space $\spa{A}$ with the same universe $A$ such that the topology on $\spa{A}$ is the discrete topology. Let $\cls{A}$ denote 
 the category whose objects are the algebras in  $\cop{ISP}(\alg{A})$ and  whose morphisms are homomorphisms of algebras. 
 Let $\mathcal{X}$ be the class of all isomorphic copies of topologically closed substructures of direct  powers (over non-empty index sets) 
of $\spa{A}$. It is easy to see that $\mathcal{X}$ with continuous maps preserving
the relational and functional structure spaces forms a category.

The structured topological space $\spa{A}$ is said to be {\em algebraic over} $\alg{A}$ if the relations and the graphs of the 
operations in $\spa{A}$ are contained in $\cop{SP}(\alg{A})$. In this case, 
contravariant functors $\XA\colon\cls{A}\to\mathcal{X}$ and 
$\AAA\colon\mathcal{X}\to\cls{A}$ may be defined as described below. First, let
$$
\XA(\alg{B})=\cls{A}[\alg{B},\alg{A}]\leq \spa{A}^{B},
$$
where $\cls{A}[\alg{B},\alg{A}]$ is the set of homomorphisms in $\cls{A}$  from $\alg{B}$ to $\alg{A}$ 
considered as a substructure and a topological subspace of $\spa{A}^{B}$, and let 
$\XA(h\colon\alg{B}\to\alg{C})$ be the map $\XA(h)\colon\XA(\alg{C})\to\XA(\alg{B})$ defined by
$$\XA(h)(f)=f\circ h\ \mbox{ for each }f\in \XA(\alg{C}).$$
Observe that if $\alg{B}$ is finite, then $\XA(\alg{B})$ is finite and carries the discrete topology. 
Similarly, let
$$
\AAA(X)=\mathcal{X}[X,\spa{A}]\leq {\alg{A}^{X}},
$$
where $\mathcal{X}[X,\spa{A}]$ is the set of morphisms from $X$ to $\spa{A}$ considered as a subalgebra of $\alg{A}^X$,
and let $\AAA(l\colon X\to Y)$ be the map $\AAA(l)\colon\AAA(Y)\to\AAA(X)$ defined by
$$
\AAA(l)(f)=f\circ l\ \mbox{ for each }f\in \AAA(Y).
$$
The evaluation map determines a natural transformation from the identity functor in $\cls{A}$ into $\AAA\circ \XA$. 
That is, for each $\alg{B}\in\cls{A}$, the map  $e_{\alg{B}}\colon\alg{B}\to \AAA(\XA(\alg{B}))$ defined by
$$
e_{\alg{B}}(b)(f)=f(b)\ \mbox{ for each } b\in\alg{B}\mbox{ and }f\in\XA(\alg{B})
$$
is a homomorphism and the class $\{e_{\alg{B}}\}_{\alg{B}\in \cls{A}}$ is a natural transformation.

If $e_{\alg{B}}$ is an isomorphism for each $\alg{B}\in\cls{A}$, then $\XA$ and $\AAA$ determine a dual equivalence between $\cls{A}$ 
and the category $\XA(\cls{A})$. In this case we say that $\spa{A}$ {\em determines a natural duality} on $\cls{A}$. It is clear from the above 
that $\spa{A}$ encodes all the information needed to reproduce the functors that determine the equivalence. If $\mathcal{X}$ coincides 
with the category of isomorphic images of $\XA(\cls{A})$, then we say that $\spa{A}$ {\em yields a full natural duality on $\cls{A}$}.

One useful byproduct of having a natural duality for a class of algebras is the following:
\begin{theorem}[{\cite[Corollary 2.2.4]{CD98}}]\label{Th:FreeDual}
  If $\spa{A}$ yields a natural duality on $\cls{A}=\cop{ISP}(\alg{A})$, then, 
  for each cardinal $\kappa$, the algebra $\AAA(\spa{A}^{\kappa})$ is free $\kappa$-generated on ${\cls{A}}$.
\end{theorem}

We remark that all the dualities considered in this paper are strong. However, rather than provide the technical definition 
of a strong duality here (see~\cite[Section~3.2]{CD98}), we instead recall two properties of strong dualities relevant 
to this paper. Namely, every strong duality is full, and embeddings (surjections)  in $\cls{A}$ correspond to surjections (embeddings) in 
$\mathcal{X}$.

For the reader unfamiliar with natural dualities, let us present the Stone duality between Boolean algebras and Stone spaces 
(zero-dimensional compact Hausdorff topological spaces) as the standard example of this approach. First, recall that the class of Boolean 
algebras $\cls{B}$ coincides with $\cop{ISP}(\alg{2})$,  where $\alg{2}$ denotes the two-element Boolean algebra. Let $\spa{2}$ be the 
$2$-element   Stone space $(\{0,1\}, \mathcal{T})$, where $\mathcal{T}$ is the discrete topology. Given a Stone space $X$, the characteristic 
map of each clopen subset is a continuous map from $X$ into $\spa{2}$. It follows that each Stone space is isomorphic to a closed subspace of 
direct  powers of $\spa{2}$, and that the functors $\XB(\alg{B})=\cls{B}[\alg{B},\alg{2}]$ and $\AAB(X)=\mathcal{X}[X,\spa{2}]$ determine a 
strong natural duality between Boolean algebras and Stone spaces. Moreover, the correspondence between $\cls{B}[\alg{B},\alg{2}]$ 
and the set of ultrafilters of $\alg{B}$ determines a natural equivalence between the functor $\XB$ and the usual presentation of Stone duality. 


\section{Admissibility}\label{s:admissibility}

In this section, we introduce the notions of admissibility, 
structural completeness, and universal completeness, and relate these to free 
algebras and the generation of varieties and positive universal classes. We then 
refine these relationships in the context of locally finite quasivarieties, obtaining 
criteria (used repeatedly in subsequent sections) for a set of  clauses or quasi-identities 
to axiomatize admissible clauses or quasi-identities relative to a quasivariety. Finally, we 
explain how natural dualities may be used to obtain these axiomatizations.


\subsection{Admissible Clauses}

Let $\Q$ be an $\lang$-quasivariety. An $\lang$-clause $\Si \imp \De$ is 
\emph{admissible in $\Q$} if for every homomorphism $\si \colon \alg{Fm_\lang} \to \alg{Fm_\lang}$:
\[
\mdl{\Q} \si(\f') \eq \si(\p') \ 	\textrm{ for all }\f' \eq \p' \in \Si  \qquad {\rm implies } \qquad \mdl{\Q} \si(\f) \eq \si(\p)\  \textrm{ for some }\f \eq \p \in \De.
\]
In particular, a negative $\lang$-clause $\Si \imp \emptyset$ is admissible in $\Q$ iff 
there is {\em no} homomorphism $\si \colon \alg{Fm_\lang} \to \alg{Fm_\lang}$ satisfying 
$\mdl{\Q} \si(\f') \eq \si(\p')$ for all $\f' \eq \p' \in \Si$.

The next result relates admissibility in $\Q$ to validity in the free algebra $\F_\Q$.

\begin{theorem} \label{t:EqAdm} 
Let $\Q$ be an $\lang$-quasivariety, let $\Si \imp \De$ be an $\lang$-clause, and let 
$\U$ be the universal class axiomatized by $\{\Si \imp \De\}$ relative to $\Q$. 
Then the following are equivalent:
\begin{itemize}
\item[{\rm (i)}]		$\Si \imp \De$ is admissible in $\Q$.
\item[{\rm (ii)}]		$\Si \mdl{\F_{\Q}} \De$.
\item[{\rm (iii)}]	$\Ugp(\Q) = \Ugp(\U)$. 
\end{itemize}
Moreover, if $|\De| = 1$, then the following is equivalent to {\rm (i)-(iii)}:
\begin{itemize}
\item[{\rm (iv)}]	$\Vg(\Q) = \Vg(\U)$.
\end{itemize}
\end{theorem}
\begin{proof}
(i) $\Rightarrow$ (ii). Suppose that $\Si \imp \Delta$ is admissible in $\Q$. Consider any homomorphism 
$g \colon \alg{Fm_\lang} \to \F_{\Q}$ satisfying $\Si \subseteq \ker g$, noting that if no such 
homomorphism exists, then trivially $\Si \mdl{\F_{\Q}} \De$. 
Let us fix $f$ to be a map sending each variable $x$ to a formula $\f$ such that $h_\Q (\f) = g(x)$, and let $\si$ be 
the unique homomorphism $\si \colon \alg{Fm_\lang} \to \alg{Fm_\lang}$ extending $f$. 
Then $h_\Q \circ \si = g$ and $\Si \subseteq \ker(h_\Q \circ \si)$. So for  $\f' \eq \p' \in \Si$, also $h_\Q(\si(\f')) =  h_\Q(\si(\p'))$ 
and, by (\ref{Eq:ValidityFree}) from Section~\ref{s:preliminaries}, $\mdl{\Q} \si(\f') \eq \si(\p')$. 
Hence,  by assumption, $\mdl{\Q} \si(\f) \eq \si(\p)$ for some $\f\eq \p\in\Delta$. Again by (\ref{Eq:ValidityFree}), 
$g(\f) = h_\Q(\si(\f)) = h_\Q(\si(\p)) = g(\p)$ as required.
 
(ii) $\Rightarrow$ (iii). Suppose that $\Si \mdl{\F_{\Q}} \De$. So $\F_\Q \in \U$ and $\Ugp(\F_\Q) \subseteq \Ugp(\U)$. 
Clearly $\Ugp(\U) \subseteq \Ugp(\Q)$, since $\U \subseteq \Q$.  Using (\ref{Eq:ValidityPos}) from Section~\ref{s:preliminaries}, 
$\Ugp(\Q) = \Ugp(\F_\Q)$. Hence $\Ugp(\Q) = \Ugp(\U)$.

 (iii) $\Rightarrow$ (i). Suppose that $\Ugp(\Q) = \Ugp(\U)$ and let $\si \colon \alg{Fm_\lang} \to \alg{Fm_\lang}$ be a homomorphism
 such that  $\mdl{\Q} \si(\f') \eq \si(\p')$ for all $\f' \eq \p' \in \Si$. Because $\Si \mdl{\U} \De$, also $\si(\Si) \mdl{\U} \si(\De)$. Hence, 
 since $\U \subseteq \Q$, it follows that $\mdl{\U} \si(\De)$. But $\Ugp(\Q)= \Ugp(\U)$, so $\mdl{\Q} \si(\De)$. 
By (\ref{Eq:ValidityPos}) from Section~\ref{s:preliminaries},  $\mdl{\Q} \si(\f) \eq \si(\p)$ for some $\f \eq \p \in \De$ as required.
 
 Note finally for (iv) that if $|\De| =1$, then $\U$ is a quasivariety and (iii) is equivalent to  
 $\Vg(\Q) = \Ugp(\U) = \Vg(\U)$. \qed
\end{proof}

\begin{example} \label{e:lattices}
The following lattice clauses  (Whitman's condition, 
meet semi-distributivity, and join semi-distributivity, respectively) 
are satisfied by all free lattices (see~\cite{Whi41,Jon61}):
\[
\begin{array}{rcl}
x \wedge y \leqn z \vee w & \imp & x \wedge y \leqn z, \ x \wedge y \leqn w, \ x  \leqn z \vee w, \ y \leqn z \vee w\\[.05in]
x \wedge y \eq x \wedge z & \imp & x \wedge y \eq x \wedge (y \vee z)\\[.05in]
x \vee y \eq x \vee z & \imp & x \vee y \eq x \vee (y \wedge z).
\end{array}
\]
Hence these clauses are admissible in the variety of lattices.
\end{example}

If $\Qg(\F_\Q)$ is axiomatized by a set of quasi-identities $\Lambda$ relative to $\Q$,  then we call  $\Lambda$ a {\em basis for 
the admissible quasi-identities of $\Q$}. Similarly, if $\Ug(\F_\Q)$  is axiomatized by a set of clauses $\Lambda$ relative to  $\Q$,  
then we call  $\Lambda$  a {\em basis for the admissible clauses of  $\Q$}.

\begin{example} \label{e:groups}
The following quasi-identities  define the quasivariety of torsion-free groups relative to groups (in a language with 
$\cdot$, $^{-1}$, and ${\rm e}$) and are satisfied by  all free groups:
\[
\underbrace{x \cdot \ldots \cdot x}_n \eq {\rm e} \quad \imp \quad x \eq {\rm e} \qquad n =2,3,\ldots.
\]
Hence these quasi-identities are admissible in the variety of groups. In fact, 
by the fundamental theorem of finitely generated abelian groups, every finitely generated 
torsion-free abelian group 
is isomorphic to a free abelian group $\alg{Z}^n$ for some $n \in \mathbb{N}$ (see, e.g.,~\cite{Fuc60}). 
These quasi-identities therefore provide a basis for the admissible clauses (and, similarly, the admissible quasi-identities) of the 
variety of abelian groups.
\end{example}

Interest in admissible rules was first stimulated largely by the phenomenon of admissible non-derivable rules in 
intuitionistic propositional logic. In particular, it was shown by Rybakov in~\cite{Ryb84} (see also the monograph~\cite{Ryb97}) that 
the set of admissible rules of this logic, corresponding to admissible quasi-identities of the variety of Heyting algebras, 
is decidable but has no finite basis. Iemhoff~\cite{Iem01} and Rozi{\`e}re~\cite{Roz92} later proved, independently, that an infinite 
basis is formed by the family of  ``Visser rules'', expressed algebraically as the quasi-identities
\[
\top \eq (\bigwedge_{i=1}^n (y_i \to x_i) \to (y_{n+1} \vee y_{n+2})) \vee z \quad \imp\quad
\top \eq \bigvee_{j=1}^{n+2} ( \bigwedge_{i=1}^n (y_i \to x_i) \to y_j) \vee z \quad  n = 2,3,\ldots.
\]
A basis for the admissible clauses of Heyting algebras is provided by the 
Visser rules together with clauses expressing, respectively, the absence of trivial algebras and the 
disjunction property:
\[
\top \eq \bot \ \imp \ \emptyset  \qquad \mbox{and} \qquad \top \eq x \vee y \ \ \imp \ \ \top \eq x, \ \top \eq y.
\]
The proof of Iemhoff (also her work on intermediate logics in~\cite{Iem05}) makes use of a characterization of 
finitely generated projective 
Heyting algebras  (retracts of finitely generated free Heyting algebras) obtained by Ghilardi~\cite{Ghi99} in the dual setting of 
Kripke frames. Similarly, 
Ghilardi's Kripke frame characterizations  of finitely generated projective algebras for various transitive modal logics~\cite{Ghi00}
were used by Je{\v r}{\'a}bek~\cite{Jer05} to obtain bases of admissible rules for these logics.  
In this paper, we will see that dual characterizations of finitely generated projective algebras can also 
play a central role in characterizing admissibility for quasivarieties generated by 
finite algebras.


\subsection{Structural Completeness and Universal Completeness}

A quasivariety $\Q$ is  \emph{structurally complete} if a quasi-identity is admissible in $\Q$ exactly 
when it is valid in $\Q$: that is, $\Si \imp \f \eq \p$ is admissible in 
$\Q$ iff $\Si \mdl{\Q} \f \eq \p$. Equivalent characterizations 
(first proved by Bergman in~\cite{Ber91}) are obtained using Theorem~\ref{t:EqAdm} 
as follows.

\begin{proposition}[{\cite[Proposition 2.3]{Ber91}}]\label{p:Bergman}
The following are equivalent for any $\lang$-quasivariety $\Q$:
\begin{itemize}
\item[{\rm (i)}] $\Q$ is structurally complete.
\item[{\rm (ii)}] $\Q = \Qg(\F_\Q)$.
\item[{\rm (iii)}] For each quasivariety $\Q' \subseteq \Q$: \ $\Vg(\Q') = \Vg(\Q)$ implies $\Q' = \Q$.
\end{itemize}
\end{proposition}
\begin{proof}
(i) $\Leftrightarrow$ (ii). Immediate from Theorem~\ref{t:EqAdm}. 

(ii) $\Rightarrow$ (iii). Suppose that   $\Q = \Qg(\F_\Q)$ and let $\Q' \subseteq \Q$ be a quasivariety satisfying 
$\Vg(\Q') = \Vg(\Q)$. Then $\F_{\Q'} = \F_{\Vg(\Q')} =  \F_{\Vg(\Q)} = \F_{\Q}$, so 
 $\Q = \Qg(\F_\Q) = \Qg(\F_{\Q'}) \subseteq \Q'$.

(iii) $\Rightarrow$ (ii).  Clearly $\Qg(\F_\Q) \subseteq \Q$ and $\Vg(\Qg(\F_\Q)) = \Vg(\Q)$. So, assuming 
(iii), $\Q = \Qg(\F_\Q)$. \qed
\end{proof}

A quasivariety $\Q$ is {\em universally complete} if a clause is admissible in $\Q$ exactly when it 
is valid in $\Q$: that is, $\Si \imp \De$ is admissible in 
$\Q$ iff $\Si \mdl{\Q} \De$. We obtain the following characterization.

\begin{proposition}\label{p:universalchar}
The following are equivalent for any $\lang$-quasivariety $\Q$:
\begin{itemize}
\item[{\rm (i)}] $\Q$ is universally complete.
\item[{\rm (ii)}] $\Q = \Ug(\F_\Q)$.
\item[{\rm (iii)}] For each universal class $\U \subseteq \Q$: \ $\Ugp(\U) = \Ugp(\Q)$ implies $\U = \Q$.
\end{itemize}
\end{proposition}
\begin{proof}
(i) $\Leftrightarrow$ (ii). Immediate from Theorem~\ref{t:EqAdm}. 

 (i) $\Rightarrow$ (iii). Suppose that $\Q$ is universally complete and let  $\U \subseteq \Q$ be a universal class with $\Ugp(\U) = \Ugp(\Q)$. 
 Consider an $\lang$-clause $\Si \imp \De$ such that $\Si \mdl{\U} \De$ and let $\U'$ be the universal class axiomatized relative to $\Q$ 
 by $\{\Si \imp \De\}$. Then $\Ugp(\Q) = \Ugp(\U) \subseteq \Ugp(\U') \subseteq \Ugp(\Q)$.  
By Theorem~\ref{t:EqAdm}, $\Si \imp \De$ is admissible in $\Q$. As $\Q$ is universally complete, $\Si \mdl{\Q} \De$. 
So $\Si \mdl{\U} \De$ implies $\Si \mdl{\Q} \De$, and hence $\Q \subseteq \U$.

(iii) $\Rightarrow$ (i). Suppose that  $\Si \imp \De$ is admissible in $\Q$ and let 
$\U$ be the universal class axiomatized by $\{\Si \imp \De\}$ relative to $\Q$. 
By Theorem~\ref{t:EqAdm}, $\Ugp(\Q) = \Ugp(\U)$. So, assuming (iii), it follows that 
$\U = \Q$ and hence $\Si \mdl{\Q} \De$.
\end{proof}

\begin{example}\label{Ex:StCompBDL}
The variety of lattice-ordered abelian groups is, like 
every quasivariety, generated as a universal class by its finitely presented members. 
But every finitely presented lattice-ordered abelian group is a retract  
of a free finitely generated lattice-ordered abelian group (that is, a finitely generated projective lattice-ordered abelian group)~\cite{Be77}. 
Hence every finitely presented lattice-ordered abelian group embeds into the free  lattice-ordered abelian group on 
countably infinitely many generators. So by Proposition~\ref{p:universalchar}, this variety is universally complete. 
\end{example}

Certain classes of algebras are very close to being universally complete; in such classes, the lack of harmony between admissibility and 
validity is due entirely to trivial algebras, or, from a syntactic perspective, to non-negative clauses.  
An $\lang$-quasivariety $\Q$ is  {\em non-negative universally complete} if a non-negative $\lang$-clause is 
admissible in $\Q$ exactly when it is valid in $\Q$: that is, $\Si \imp \{\f \eq \p\} \cup \De$ is admissible in 
$\Q$ iff $\Si \mdl{\Q} \{\f \eq \p\} \cup \De$. We obtain the following characterization:

\begin{proposition}\label{p:suniversalchar}
The following are equivalent for any $\lang$-quasivariety $\Q$:
\begin{itemize}
\item[{\rm (i)}] $\Q$ is non-negative universally complete.
\item[{\rm (ii)}] Every $\lang$-clause admissible  in $\Q$ is satisfied by all non-trivial members of $\Q$. 
\item[{\rm (iii)}] Every non-trivial algebra of $\Q$ is in  $\Ug(\F_\Q)$.
\end{itemize}
\end{proposition}
\begin{proof}
 (i) $\Rightarrow$ (ii). 
 Suppose that $\Q$ is non-negative universally complete and 
 let $\Si \imp \De$ be an $\lang$-clause admissible in $\Q$. If $\De \neq \emptyset$, then, by assumption, $\Si  \mdl{\Q} \De$.  
Suppose now that $\De = \emptyset$ and let $x,y$ be variables not occurring in $\Si$. 
Then $\Si \imp x \eq y$ is admissible in $\Q$. Again, by assumption, 
$\Si \mdl{\Q} x \eq y$.  If $\alg{A}\in\Q$ is non-trivial, then $\Si \mdl{\alg{A}} x\eq y$   iff 
$\Si \mdl{\alg{A}} \emptyset$. 
So $\Si \imp \emptyset$ is satisfied by all non-trivial members of $\Q$.

 (ii) $\Rightarrow$ (i). Follows immediately since any non-negative $\lang$-clause is satisfied by all trivial $\lang$-algebras.
 
 (ii) $\Leftrightarrow$ (iii). By Theorem~\ref{t:EqAdm}, an $\lang$-clause is admissible in $\Q$  iff it 
is valid in $\Ug(\F_\Q)$, 
 so (ii) is equivalent to the statement that $\Ug(\F_\Q)$ contains all non-trivial members of $\Q$, i.e., (iii). 
\end{proof}

\noindent
Any non-negative universally complete $\lang$-quasivariety $\Q$ is  structurally complete. Moreover, if  
 $\Q$ is not universally complete, then by Proposition~\ref{p:suniversalchar},  {\em any} negative $\lang$-clause 
 not satisfied by trivial $\lang$-algebras provides a basis for the admissible $\lang$-clauses of $\Q$.

\begin{example}\label{e:Bool}
Each non-trivial Boolean algebra, but no trivial Boolean algebra, embeds into a free Boolean algebra. Hence the
variety of Boolean algebras is non-negative universally complete, but not universally complete, and 
$\bot \eq \top \ \imp \ \emptyset$ forms a basis for its admissible clauses. The same holds for Boolean lattices 
(without $\bot$ and $\top$ 
in the language), but in this case, a suitable basis is formed by $x \land \lnot x \eq x \vee \lnot x \ \imp\ \emptyset$. 
In Section~\ref{s:distributive}, we will see that  bounded distributive lattices provide an example of a variety that 
is structurally complete, but neither universally complete nor non-negative universally complete. 
\end{example}


\subsection{Local Finiteness}

The quasivarieties considered in subsequent sections of this paper are all {\em locally finite}: that is, their finitely generated members are finite. 
This property has important consequences for characterizations of the universal classes and quasivarieties generated by free algebras, 
which are in turn crucial for our study of admissible clauses and quasi-identities in these classes. In particular, we will make frequent 
use of the following result:

\begin{lemma}\label{l:localfiniteembeddings}
For any locally finite $\lang$-quasivariety $\Q$:
\begin{itemize}
\item[{\rm (a)}]	$\alg{A} \in \Ug(\F_\Q)$ iff each finite subalgebra $\alg{B}$ of $\alg{A}$ is in $\cop{IS}(\F_\Q)$. 	
\item[{\rm (b)}]	$\alg{A} \in \Qg(\F_\Q)$ iff each finite subalgebra $\alg{B}$ of $\alg{A}$ is in $\cop{ISP(\F_\Q)}$.		
\end{itemize}
Moreover, if $\Q = \Qg(\K)$ for some finite class $\K$ of finite $\lang$-algebras, then
\begin{itemize}
\item[{\rm (c)}]		$\Qg(\F_\Q) = \cop{ISP}(\F_\Q)$.	
\end{itemize}
\end{lemma}
\begin{proof}
First observe that, since universal classes are axiomatized by clauses, if $\U$ is a universal class and $\alg{A}$ is an algebra of the same language, 
then $\alg{A}\in\U$ iff each finitely generated subalgebra of $\alg{A}$ is in $\U$. Under the assumption of local finiteness, this is equivalent to 
each finite subalgebra of $\alg{A}$ being a member of $\U$. Hence, to establish (a) and (b), it suffices to prove: 

\begin{itemize}
\item[{\rm (a')}]	Each finite $\alg{B} \in \Ug(\F_\Q)$ is in $\cop{IS}(\F_\Q)$.	
\item[{\rm (b')}]	Each finite $\alg{B}\in\Qg(\F_\Q)$ is in $\cop{ISP(\F_\Q)}$.		
\end{itemize}
(a') Let $B = \{b_1,\ldots, b_n\}$ be an enumeration of the elements of $\alg{B} \in \Ug(\F_\Q)$. By assumption, there exists a 
non-empty set $I$, an ultrafilter $U$ over $I$, and an embedding $f\colon \alg{B}\to \F_{\Q}^I/U$.  Let $\{g_1,\ldots,g_n\}\in \F_{\Q}^I$ be such that $f(b_i)=[g_i]_{U}$ for each $1\leq i\leq n$.
Because $f$ is one-to-one and $U$ is an ultrafilter of $I$, for each  $1\leq j< k\leq n$, the set 
$I_{jk} = \{i \in I\mid   g_j(i)\neq g_k(i)\}$ is in $U$. 
Similarly, since $f$ is a homomorphism, for each $m$-ary operation $h$ of $\lang$ and each $j_1,\ldots,j_m\in\{1,\ldots,n\}$, if $h(b_{j_1},\ldots, b_{j_m})=b_k$, the set $I_{h,j_1,\ldots,j_m}=\{i\in I\mid h(g_{j_1}(i),\ldots,g_{j_m}(i))=g_k(i)\}$ is in $U$.
Now, since $\lang$ is finite,
\[
I'=\bigcap \{I_{jk} \mid  1\leq j< k\leq n \} \cap \bigcap \{I_{h,j_1,\ldots,j_m} \mid  h\mbox{ an $m$-ary operation of $\lang$}, 
1\leq j_1,\ldots,j_m\leq n \}\in U.
\]
So there exists $i\in I'$ and it follows that the map $f'\colon B\to \F_{\Q}$ defined by $f(b_j)=g_j(i)$ is a well-defined one-to-one homomorphism.

(b') Let $\alg{B}$ be a finite algebra in $\Qg(\F_\Q)=\cop{ISPP}_{U}(\F_\Q)$. Then there exists  a set $\K$ of algebras in 
$\cop{P}_{U}(\F_\Q)$ and an embedding $h\colon\alg{B}\to\prod \K$. For each $\alg{C}\in \K$ let $\pi_{\alg{C}}\colon\prod \K\to \alg{C}$ 
be the projection map. Then $\K'=\{\pi_{\alg{C}}\circ h(\alg{B})\mid \alg{C}\in\K \}$ satisfies $\alg{B}\in\cop{IS}(\prod \K')$ and $\K'\subseteq\cop{SP}_U(\F_\Q)$. 
As $\alg{B}$ is finite, each member of $\K'$ is finite. By (a'), $\K'\subseteq\cop{IS}(\F_\Q)$. So $\alg{B}\in\cop{ISPS}(\F_\Q)= \cop{ISP}(\F_\Q)$.

(c) Suppose that $\Q = \Qg(\K)$ for some finite class $\K$ of finite $\lang$-algebras and let 
$n$ be the maximum of the cardinalities of the members of $\K$. Then each member of $\K$ 
is a homomorphic image of $\F_\Q(n)$. So $\Vg(\F_\Q(n)) = \Vg(\K)$ and hence $\F_\Q \in \Qg(\F_\Q(n))$ 
and $\Qg(\F_\Q) = \Qg(\F_\Q(n))$. But since $\K$ is a finite class of finite algebras, $\F_\Q(n)$ is finite. 
So $\Qg(\F_\Q) = \Qg(\F_\Q(n)) = \cop{ISPP}_U(\F_\Q(n)) = \cop{ISP}(\F_\Q(n)) = \cop{ISP}(\F_\Q)$. 
\end{proof}

Combining this result with Propositions~\ref{p:Bergman} and~\ref{p:universalchar}, we also obtain:

\begin{corollary} \label{c:strcomp}
For any locally finite $\lang$-quasivariety $\Q$:
\begin{itemize}
\item[{\rm (a)}]	$\Q$ is universally complete iff each finite algebra in $\Q$ is in $\cop{IS}(\F_\Q)$.
\item[{\rm (b)}]	$\Q$ is structurally complete iff each finite algebra in $\Q$ is in $\cop{ISP}(\F_\Q)$.
\end{itemize}
Moreover, if $\Q = \Qg(\K)$ for some finite class $\K$ of finite $\lang$-algebras, then
\begin{itemize}
\item[{\rm (c)}]	$\Q$ is structurally complete iff  $\Q=\cop{ISP}(\F_\Q)$.
\end{itemize}
\end{corollary}

To find a basis for the admissible clauses of a locally finite quasivariety $\Q$, it suffices, using Lemma~\ref{l:localfiniteembeddings}, 
to find an axiomatic characterization of the finite subalgebras of $\F_\Q$. More precisely:

\begin{lemma}\label{l:mainlemma}
Let $\Q$ be a locally finite $\lang$-quasivariety.
\begin{itemize}
\item[{\rm (a)}] 
The following are equivalent for any set of $\lang$-clauses $\Lambda$:
\begin{itemize}
\item[{\rm (i)}]	 For each finite $\alg{B} \in \Q$: $\alg{B} \in \cop{IS}(\F_\Q)$ iff  $\alg{B}$ satisfies $\Lambda$. 
\item[{\rm (ii)}] 	$\Lambda$ is a basis for the admissible clauses of $\Q$.
\end{itemize}		
\item[{\rm (b)}] 
The following are equivalent for any set of $\lang$-quasi-identities $\Lambda$:
\begin{itemize}
\item[{\rm (iii)}]	 For each finite $\alg{B} \in \Q$: $\alg{B} \in \cop{ISP}(\F_\Q)$ iff $\alg{B}$ satisfies $\Lambda$. 
\item[{\rm (iv)}] 	$\Lambda$ is a basis for the admissible quasi-identities of $\Q$.
\end{itemize}
\end{itemize}	
\end{lemma}
\begin{proof}
(a)  (i) $\Rightarrow$ (ii). It suffices to observe that for each $\alg{A} \in \Q$:\\[.1in]
\noindent\begin{tabular}{rll}
$\alg{A}$ satisfies $\Lambda$ \ iff & each finite subalgebra of $\alg{A}$ satisfies $\Lambda$  	&($\Q$ is locally finite)\\
				 iff & each finite subalgebra of $\alg{A}$ is in $\cop{IS}(\F_\Q)$ 	&(assumption)\\
  				 iff &   $\alg{A} \in \Ug(\F_\Q)$					&(Lemma~\ref{l:localfiniteembeddings}(a)). 
\end{tabular}\\[.1in]
(ii) $\Rightarrow$ (i). Suppose that $\Lambda$ is a basis for the admissible clauses of $\Q$ and 
consider a finite $\alg{B} \in \Q$. By assumption, $\alg{B}$ satisfies $\Lambda$ iff $\alg{B} \in \Ug(\F_\Q)$. 
But then by Lemma~\ref{l:localfiniteembeddings}(a), $\alg{B} \in \Ug(\F_\Q)$ iff $\alg{B} \in \cop{IS}(\F_\Q)$.

(b) Very similar to (a). 
 \end{proof}


We conclude this section by relating the problem of finding bases for the admissible clauses and quasi-identities of 
a quasivariety $\Q$ generated by a finite algebra $\alg{A}$ to natural dualities. 
Suppose that there is a structure $\spa{A}$ that yields a strong natural duality on $\Q$. To obtain a basis for the 
admissible clauses of $\Q$, we seek a set of clauses $\Lambda$ that characterizes the finite algebras of $\Q$ that 
embed into $\F_\Q$. But these algebras correspond on the dual side to images of finite powers of 
$\spa{A}$ under morphisms of the dual category.  Hence we first seek conditions $C$ on  dual spaces 
(in the cases considered in this paper,  first-order conditions) to be images of finite powers of $\spa{A}$  under morphisms 
of the dual category. We then seek a set of clauses $\Lambda$ such that a finite 
algebra $\alg{B} \in \Q$ satisfies the clauses in $\Lambda$ iff its dual space satisfies the conditions $C$.
 In this way, we avoid confronting the characterization of subalgebras of the free algebra $\F_\Q$ directly and consider rather  
 the range of morphisms of the combinatorial structures $\spa{A}^n$ for $n \in \mathbb{N}$. 
 Similarly, to obtain a basis for the admissible quasi-identities of $\Q$, we seek a set of quasi-identities $\Lambda$ that 
 characterizes the finite algebras of $\Q$ that embed into a finite power of $\F_\Q$, and therefore corresponding conditions 
 on  dual spaces to be images of finite copowers of finite powers of $\spa{A}$ under morphisms of the dual category.


\section{Distributive Lattices and Stone Algebras}\label{s:distributive}

In this section, we investigate and find bases for the admissible clauses and quasi-identities of (bounded) distributive lattices 
and Stone algebras. For the sake of uniformity and also as preparation for the more involved cases of Kleene and De~Morgan algebras, 
we make use here of natural dualities. Let us remark, however, that in this setting direct algebraic proofs characterizing subalgebras of (powers of) 
free algebras are also relatively straightforward.


\subsection{Distributive Lattices} 

The class of bounded distributive lattices $\BDL$ is a variety and coincides with $\cop{ISP}(\alg{2})$, where
$\alg{2}=(\{0,1\},\min,\max,0,1)$~(see, e.g.,~\cite{BS81}). 
Clearly $\alg{2}\in \cop{IS}(\F_{\BDL})$, so also $\BDL=\cop{ISP}(\alg{2})\subseteq\Qg(\F_{\BDL})$ and, by 
Proposition~\ref{p:Bergman}:

\begin{proposition}
 The variety of bounded distributive lattices is structurally complete.  
 \end{proposition}
 
 \noindent
However, $\BDL$ is neither universally complete nor non-negative universally complete; that is, 
there are  non-negative clauses that are admissible but not satisfied by all members of $\BDL$. 
Here, we provide a basis for the admissible clauses of  $\BDL$, making use 
of the well-known natural duality between bounded distributive lattices and Priestley spaces 
(see~\cite{Pri70} and~\cite[Theorem~4.3.2]{CD98}).

Recall that a structured topological space $(X,\leq,\tau)$ is a {\em Priestley space} if $(X,\tau)$ is a compact topological space, 
$(X,\leq)$ is a partially ordered set, and for all $x,y\in X$ with $x\nleq y$,  there exists a clopen upset $U\subseteq X$ such that $x\in U$ and $y\notin U$.   
The structured topological space $\spa{2}=(\{0,1\},\leq_2,\mathcal{P}(\{0,1\}))$, where $\leq_2$ is the partial order on $\{0,1\}$ such that $0\leq_2 1$, 
determines a strong natural duality between the variety $\BDL$ and the category of Priestley spaces. In particular, by 
 Theorem~\ref{Th:FreeDual} applied to bounded distributive lattices, 
the free $n$-generated bounded distributive lattice is $\ABDL(\spa{2}^n)$  for $n \in \mathbb{N}$.
Moreover, a finite bounded distributive lattice $\alg{L}$ is a retract of a free distributive lattice (that is, $\alg{L}$ is projective) 
iff the poset $\XBDL(\alg{L})$ is a lattice (see~\cite{BH70}).

The following lemma provides the key step in our method for obtaining a basis for the admissible clauses of 
bounded distributive lattices (Theorem~\ref{Th:BasisBDL}).

\begin{lemma}\label{Lemma_ISPU-BDL}
Let $\alg{L}$ be a finite bounded distributive lattice. Then the following are equivalent:
\begin{itemize}
\item[{\rm (i)}] $\alg{L}\in \cop{IS}(\F_{\BDL})$.
\item[{\rm (ii)}] $\XBDL(\alg{L})$ is a non-empty bounded poset.
\item[{\rm (iii)}] $\alg{L}$ satisfies the following clauses:
  \begin{eqnarray}
     \top\eq\bot &\imp&\emptyset 				\label{Eq:nontrivial}\\
    x\wedge y\eq \bot &\imp& x\eq \bot,\ y\eq\bot 	\label{Eq:0prime}\\
    x\vee y\eq \top&\imp& x\eq \top, \ y\eq\top.		\label{Eq:1prime}
  \end{eqnarray}
\end{itemize}
\end{lemma}
\begin{proof}
(i) $\Rightarrow$ (ii). Since $\alg{L}$ is finite,  $\alg{L}\in \cop{IS}(\F_{\BDL}(n))$ for some  $n\in\mathbb{N}$. 
Let $h\colon \alg{L}\to \F_{\BDL}(n)$ be a one-to-one homomorphism.  $\spa{2}$ yields a strong 
duality, so $\XBDL(h)\colon \XBDL(\F_{\BDL}(n))\to \XBDL(\alg{L})$ 
is an onto order-preserving map. But $\XBDL(\F_{\BDL}(n))$ is isomorphic to $\spa{2}^{n}$, so it is bounded, and 
hence $\XBDL(\alg{L})$ is also bounded.

(ii) $\Rightarrow$ (i). Let us denote  the top and bottom maps in $\XBDL(\alg{L})=(X,\leq,\mathcal{P}(X))$ by $t$ and $s$, respectively.  
Let  $\leq'$ denote the partial order on $X$ defined as follows:
$$
x\leq' y\qquad \mbox{ iff } \qquad x=s  \quad \mbox{or}\quad y=t \quad \mbox{or}\quad x=y.
$$
Clearly $(X,\leq')$ is a lattice. Hence $\ABDL(X, \leq',\mathcal{P}(X))$ is a retract of a finitely generated free distributive lattice, 
and it follows that $\ABDL(X, \leq',\mathcal{P}(X))\in\cop{IS}(\F_{\BDL})$.

Since $\leq'\subseteq \leq$, the identity map $\id_X$ on $X$ is an order-preserving continuous map from  $(X, \leq',\mathcal{P}(X))$ 
onto $\XBDL(\alg{L})$. But $\spa{2}$ determines a strong natural duality, so the map 
$\ABDL(\id_X)\circ e_{\alg{L}}\colon \alg{L}\to \ABDL(X, \leq',\mathcal{P}(X))$ 
is one-to-one. Hence $\alg{L}\in \cop{IS}(\F_{\BDL})$.

(ii) $\Rightarrow$ (iii). Since $\XBDL(\alg{L})$ is non-empty, $\alg{L}$ is non-trivial and satisfies \eqref{Eq:nontrivial}. 
Suppose now that  $s\in\XBDL(\alg{L})$ is the bottom map of $\XBDL(\alg{L})$, 
and that $c,d\in L$ are such that $c\vee d=\top$. Then $e_{\alg{L}}(c\vee d)(s)=e_{\alg{L}}(c)(s)\vee e_{\alg{L}}(d)(s)=1\in \alg{2}$, so 
$e_{\alg{L}}(c)(s)=1$ or $e_{\alg{L}}(d)(s)=1$. Suppose that $e_{\alg{L}}(c)(s)=1$. Since $s$ is the bottom element of $\XBDL(\alg{L})$ and 
$e_{\alg{L}}(c)$ is a monotone map, $e_{\alg{L}}(c)(z)=1$ for each $z\in\XBDL(\alg{L})$; equivalently, $c=\top$. Similarly, $d=\top$ if $e_{\alg{L}}(d)(s)=1$. 
So $\alg{L}$ satisfies \eqref{Eq:1prime}. Similarly, $\alg{L}$ satisfies \eqref{Eq:0prime}  since $\XBDL(\alg{L})$ has an upper bound.

(iii) $\Rightarrow$ (ii). By  \eqref{Eq:1prime} and \eqref{Eq:nontrivial}, the map $h$ defined by $h(\top)=1$ and $h(x)=0$ for each $x\neq \top$ 
is a homomorphism from $\alg{L}$ into $\alg{2}$ and is the bottom element of $\XBDL(\alg{L})$. 
Similarly, using \eqref{Eq:0prime} and \eqref{Eq:nontrivial}, the map $k$ defined by 
$k(\bot)=0$ and $k(x)=1$ if $x\neq\bot$ is the top element of $\XBDL(\alg{L})$. 
So  $\XBDL(\alg{L})$ is a non-empty bounded poset. 
\end{proof}

Notice the key intermediary role played here by the natural duality in characterizing finite subalgebras of $\F_{\BDL}$ using the clauses 
 \eqref{Eq:nontrivial}, \eqref{Eq:0prime}, and \eqref{Eq:1prime}. On one hand, it is shown that satisfaction of the clauses by the finite algebra $\alg{L}$ 
corresponds to a certain characterization of $\XBDL(\alg{L})$. On the other hand, it is also shown that this characterization of $\XBDL(\alg{L})$ 
corresponds  to $\alg{L} \in \cop{IS}(\F_{\BDL})$. Observe, moreover, that this second step makes use of the 
dual objects of a special subclass of  $\cop{IS}(\F_{\BDL})$: in this case, the finite projective distributive lattices. 
The structure of this proof will be repeated throughout the remainder of this paper. 

Now, immediately, using Lemmas~\ref{l:mainlemma} and~\ref{Lemma_ISPU-BDL}:

\begin{theorem}\label{Th:BasisBDL}
 $\{\eqref{Eq:nontrivial}, {\eqref{Eq:0prime}}, {\eqref{Eq:1prime}}\}$ is  a basis for the admissible clauses of bounded distributive lattices.
\end{theorem}

By making a  detour through the case of bounded distributive lattices, 
we can also show that the variety $\cls{DL}$ of distributive lattices is universally complete.

\begin{theorem}
The variety of distributive lattices is universally complete.
\end{theorem}
\begin{proof}
First we describe a simple construction that will be 
employed with slight modifications  in Sections~\ref{Sec:DemLat}~and~\ref{Sec:KleLat}. 
 Given a distributive lattice $\alg{L}$, let $\overline{\alg{L}}$ 
 denote the bounded distributive lattice obtained from $\alg{L}$ by adding fresh 
 top and bottom elements $\top,\bot$, respectively.
   
By  Corollary~\ref{c:strcomp}(a), to establish the universal completeness of $\cls{DL}$, it suffices to show that each finite 
distributive lattice $\alg{L}$ can be embedded into $\F_{\cls{DL}}$. Observe, however, that $\overline{\alg{L}}$ 
clearly satisfies the clauses~\eqref{Eq:nontrivial}, \eqref{Eq:0prime},  and \eqref{Eq:1prime}, so by 
Lemma~\ref{Lemma_ISPU-BDL},  $\overline{\alg{L}}\in\cop{IS}({\F_{\BDL}})$. 
But $\overline{\F}_{\cls{DL}}\cong\F_{\BDL}$, so  $\alg{L}\in\cop{IS}({\F_{\cls{DL}}})$ 
as required. 

\end{proof}


\subsection{Stone algebras}

An algebra $(A,\vee,\wedge, ^{*}\!,\bot,\top)$ is called a {\em pseudocomplemented distributive lattice} if $(A,\vee,\wedge,\bot,\top)$ is a bounded 
distributive lattice and $a^{*}= \max \{b\in A \mid a\wedge b=\bot\}$ for each $a\in A$. A pseudocomplemented distributive lattice is 
 a {\em Stone algebra} if, additionally, $a^{*}\vee  a^{**}= \top$ for all $a \in A$. 
The class $\cls{ST}$ of Stone algebras forms a variety whose unique proper non-trivial subvariety is the variety of Boolean algebras. 
We have already seen that the variety of Boolean algebras is non-negative universally complete but not universally complete. In this section we will 
show that the same holds for $\cls{ST}$. 

Let $\alg{S}=(\{0,a,1\},\wedge,\vee,^{*},0,1)$ denote the Stone algebra with order  $0<a<1$
and $0^{*}=1$, $1^{*}=0$, and $a^{*}=0$. Then $\cls{ST} = \cop{ISP}(\alg{S})$ (see, e.g.,~\cite{Gra98}). 
Consider now the structure $\spa{S}=(\{0,a,1\},\leq_S,d,\mathcal{P}(\{0,a,1\}))$ where $\leq_S$ is a partial order with a unique non-trivial 
edge $1\leq_S a$ and $d\colon \{0,a,1\}\to  \{0,a,1\}$ is defined by $d(a)=d(1)=1$ and $d(0)=0$ (see Figure~\ref{Fig:Stone}(a)). 
Then $\spa{S}$ determines a 
strong natural duality between  $\cls{ST}$ and the class of Priestley spaces endowed with a continuous map that sends each 
element $x$ to the unique minimal element $y$ satisfying $y \le x$
(see~\cite{Dav82} and~\cite[Theorem 4.3.7]{CD98}). 
By Theorem~\ref{Th:FreeDual}, the dual space of the free $n$-generated Stone algebra $\XST(\F_{\cls{ST}}(n))$ is isomorphic to 
$\spa{S}^{n}$ (Figure~\ref{Fig:Stone}(b) depicts the case $n=2$).

\begin{figure}
\begin{center}
\begin{pspicture}(0,-.7)(12,4)
\psdots(1,1)(2,1)(2,2)
\psline(2,1)(2,2)
\pscurve[linecolor=gray,arrows=<-](2.15,1.1)(2.3,1.5)(2.15,1.9)
\psarc[linecolor=gray,arrows=->](0.8,1.1){.2}{20}{310}
\psarc[linecolor=gray,arrows=->](1.75,1.1){.2}{20}{310}\rput[c](1,.6){$0$}
\rput[c](2,.6){$1$}
\rput[c](2,2.4){$a$}
\rput[c](1.5,-.5){(a) $\spa{S}$}
\psdots(5,1)(6.2,1)(6.2,2)(7.4,1)(7.4,2)(9.8,1)(8.6,2)(11,2)(9.8,3)
\psarc[linecolor=gray,arrows=->](4.8,1.1){.2}{20}{310}
\psline(6.2,1)(6.2,2)
\pscurve[linecolor=gray,arrows=<-](6.35,1.1)(6.5,1.5)(6.355,1.9)
\psarc[linecolor=gray,arrows=->](5.95,1.1){.2}{20}{310}
\psline(7.4,1)(7.4,2)
\pscurve[linecolor=gray,arrows=<-](7.55,1.1)(7.7,1.5)(7.55,1.9)
\psarc[linecolor=gray,arrows=->](7.15,1.1){.2}{20}{310}
\psline(9.8,1)(11,2)(9.8,3)(8.6,2)(9.8,1)
\psline[linecolor=gray,arrows=<-](9.8,1.15)(9.8,2.85)
\pscurve[linecolor=gray,arrows=->](8.6,1.85)(9.1,1.2)(9.65,1)
\pscurve[linecolor=gray,arrows=->](11,1.85)(10.5,1.2)(9.95,1)
\psarc[linecolor=gray,arrows=->](9.8,0.75){.2}{120}{430}
\rput[c](5,.6){$(0,0)$}
\rput[c](6.2,.6){$(0,1)$}
\rput[c](6.2,2.4){$(0,a)$}
\rput[c](7.4,.6){$(1,0)$}
\rput[c](7.4,2.4){$(a,0)$}
\rput[c](9.8,.3){$(1,1)$}
\rput[c](8.45,2.4){$(a,1)$}
\rput[c](11.15,2.4){$(1,a)$}
\rput[c](9.8,3.3){$(a,a)$}
\rput[c](8,-.5){(b) $\spa{S}^2$}

\end{pspicture}
\caption{The spaces $\spa{S}$ and $\spa{S}^2$ ($d$ is represented by arrows) }\label{Fig:Stone}
\end{center}
\end{figure}
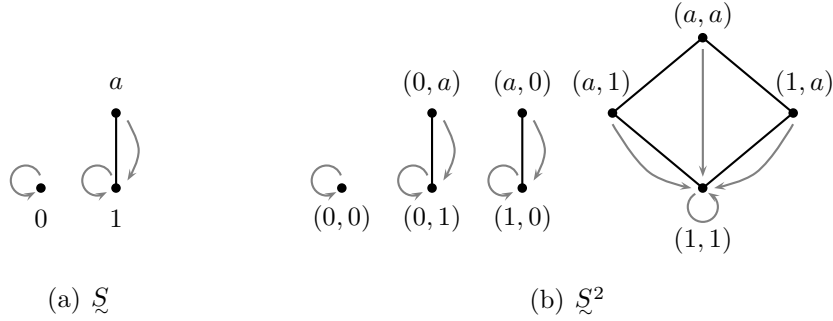

Let us fix some notation. Given a poset $P$, let $\max(P)$ and $\min(P)$ denote the set of 
maximal and minimal elements of $P$, respectively. Trivially, if  $P$ is finite, then $\max(P)$ and $\min(P)$ are non-empty.

\begin{lemma}\label{L:nontrivial-Stone}
   A finite Stone algebra belongs to $\cop{IS}(\F_{\cls{ST}})$ if and only if it is non-trivial. 
   
\end{lemma}
\begin{proof}
Since $\F_{\cls{ST}}$ is non-trivial, $\top\neq\bot$ in $\F_{\cls{ST}}$ and each member of 
$\cop{IS}(\F_{\cls{ST}})$ must contain at least two elements. This proves the left to right implication. 

For the converse, first observe that $\max(\spa{S}^{n})=\{0,a\}^n$. So $\max(\spa{S}^{n})$ has $2^{n}$ elements 
and only $(0,\ldots,0)$ among those elements satisfies $d(x)=x$. Let $\alg{A}$ be a finite non-trivial Stone algebra
and let $\XST(\alg{A})=(\{x_1,\ldots,x_k\},\leq,d,\mathcal{P}(\{x_1,\ldots,x_k\}))$ be its dual space. Fix $n$ such that $k\leq 2^{n}-1$, and 
let $\{m_1,\ldots,m_{2^{n}}\}$ be an enumeration of the maximal elements of  $\spa{S}^{n}$, such that $m_{2^{n}}=(0,\ldots,0)$. 
Fix $z\in \XST(\alg{A})$ such that $d(z)=z$, and let $\eta\colon\spa{S}^{n}\to \XST(\alg{A})$ be defined as follows:
$$
\eta(y)=
\begin{cases}
 x_i		& {\rm if } \ y=m_i \ {\rm for \ some } \ i\in\{1,\ldots,k\}\\
 d(x_i)	& {\rm if } \ y<m_i \ {\rm for \ some } \ i\in\{1,\ldots,k\}\\
 z		& {\rm otherwise.}
\end{cases}
$$
 Then $\eta$ is a monotone onto map that preserves $d$. Hence 
 $\AST(\eta)\circ e_{\alg{A}}\colon \alg{A}\to \AST(\spa{S}^{n})$ is an embedding. 
 Since $\AST(\spa{S}^{n})$ is isomorphic to $\F_{\cls{ST}}(n)$, the result follows. 
\end{proof}

We then obtain the following:

\begin{theorem}
  The variety of Stone algebras is non-negative universally complete but not universally complete, and 
  $\{\eqref{Eq:nontrivial}\}$ is a basis for the admissible clauses of this variety.
\end{theorem}
\begin{proof}
To prove that the variety $\cls{ST}$ of Stone algebras is non-negative universally complete, it suffices, by  
Proposition~\ref{p:suniversalchar}, to check that every non-trivial Stone algebra $\alg{A}$ is in $\Ug(\F_{\cls{ST}})$.  
This is true when $\alg{A}$ is finite by  Lemma~\ref{L:nontrivial-Stone}, and hence true in general by 
 Lemma~\ref{l:localfiniteembeddings}. $\cls{ST}$  fails to be universally complete since $\eqref{Eq:nontrivial}$ is 
 admissible but not  satisfied by trivial Stone algebras, and this clause therefore provides a basis for the admissible clauses 
 of the variety. 
 \end{proof}


\section{De~Morgan Algebras and Lattices}\label{s:demorgan}

A \emph{De~Morgan lattice} $(A,\wedge,\vee,\neg)$ is a distributive lattice
equipped with an additional unary operation $\neg$ that is involutive ($a=\neg \neg a$ for 
all $a \in A$) and satisfies the De~Morgan law $\lnot (a \wedge b) = \neg a \vee \neg b$ for all $a,b \in A$. 
A \emph{De~Morgan algebra} $(A,\wedge,\vee,\neg,\bot,\top)$ consists of a bounded De~Morgan lattice with 
additional constants $\bot$ and $\top$ for the bottom and top elements of the lattice, respectively. 

Our aim in this section is to provide bases for the admissible clauses and quasi-identities of the 
varieties $\cls{DMA}$ of De~Morgan algebras and $\cls{DML}$ of De~Morgan lattices. 
More precisely, consider the following clauses, recalling that $\f \leqn \p$ stands for 
$\f \wedge \p \eq \f$:
\begin{align}
\nonumber x\vee y \eq \top 					& \quad \imp \quad  x\eq \top ,\  y\eq \top \tag{\ref{Eq:1prime}}\\
 x\eq\lnot x										&  \quad\imp\quad\emptyset	\label{Eq:DMnontrivial}\\
x\eq\neg x 										& \quad \imp\quad x\eq y \label{deMorganrule}\\
x\leqn \neg x, \ \neg(x\vee y)\leqn x\vee y, \ \neg y \vee z\eq\top	
													& \quad \imp \quad z\eq\top\label{Eq_DeM1}\\
x\leqn \neg x, \ y\leqn \neg y, \ x\wedge y\eq \bot			
													& \quad \imp \quad x\vee y\leqn \neg(x\vee y)\label{Eq_DeM2}.
\end{align} 
We prove the following:

\begin{itemize}

\item The clauses (\ref{Eq:1prime}) and (\ref{Eq:DMnontrivial}) form a basis for the admissible clauses of De~Morgan algebras 
		(Theorem~\ref{t:BasisDeMorgan}). 
		
\item The quasi-identities (\ref{Eq_DeM1}) and (\ref{Eq_DeM2}) form a basis for the admissible quasi-identities of De~Morgan algebras (Theorem~\ref{t:BasisDeMorganQ}).

\item	The clause  (\ref{Eq:DMnontrivial})  forms a basis for the admissible clauses of De~Morgan lattices (Theorem~\ref{t:DeMorganLattices}).

\item The quasi-identity  (\ref{deMorganrule})  forms a basis for the admissible quasi-identities of De~Morgan lattices (Theorem~\ref{t:DeMorganLattices}).
\end{itemize}

\noindent
In~\cite{MR12} it is shown that (\ref{deMorganrule}) forms a basis for the admissible quasi-identities 
of De~Morgan lattices, and that a quasi-identity is admissible in De~Morgan algebras 
iff it is valid in all De~Morgan algebras satisfying (\ref{Eq:1prime}) and (\ref{deMorganrule}). 
 This proof makes use of a description of the  finite lattice of quasivarieties of De Morgan lattices 
 provided by Pynko in~\cite{Pyn99}. However, it is  the dual perspective on De~Morgan algebras presented below 
that allows us to  obtain a comprehensive description of admissibility in these varieties. 


\subsection{A Natural Duality for De~Morgan Algebras}

We obtain bases for the admissible clauses and quasi-identities of De Morgan algebras (and, with a little extra work, also for 
De Morgan lattices) by characterizing the dual spaces of the algebras in $\cop{IS}(\F_{\cls{DMA}})$ and 
$\cop{ISP}(\F_{\cls{DMA}})$ according to the following natural duality.

Consider first the De~Morgan algebra $\alg{D}=(\{0,a,b,1\},\wedge,\vee,\neg,0,1)$ with an order given by the 
Hasse diagram in Figure~\ref{Fig:Demorgan}(a) and $\neg 0=1$, $\neg 1=0$, $\neg a=a$, and $\neg b=b$.
Then  $\cls{DMA} = \Qg(\alg{D})$ (see~\cite{Kal58}). 
Now let $\spa{D}=(\{0,a,b,1\},\leq_D,f,\mathcal{P}(\{0,a,b,1\}))$ where $\leq_D$ is described by the Hasse diagram in Figure~\ref{Fig:Demorgan}(b) and $f(0)=0$, $f(1)=1$, $f(a)=b$, and $f(b)=a$. 

A structure $( X,\leq_X,f_X, \tau_X) $ is called a {\em De Morgan space} if
\begin{itemize}
\item $( X,\leq_X, \tau_X)$ is a Priestley space; 
\item $f_X$ is an involutive order-reversing homeomorphism.
\end{itemize}

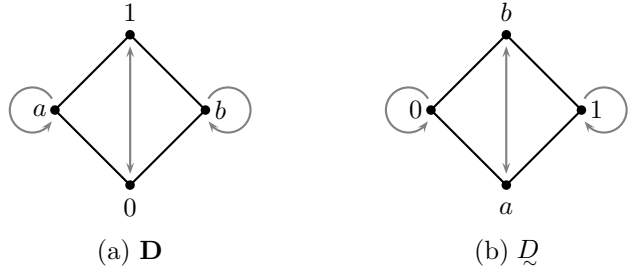
\begin{figure}
\begin{center}
\begin{pspicture}(0,0)(9,5)
\psdots(2,1)(1,2)(2,3)(3,2)
\psline(2,1)(1,2)(2,3)(3,2)(2,1)
\psline[linecolor=gray,arrows=<->](2,1.15)(2,2.85)
\psarc[linecolor=gray,arrows=->](0.7,2){.3}{30}{330}
\psarc[linecolor=gray,arrows=<-](3.3,2){.3}{210}{510}
\rput[c](2,.7){$0$}
\rput[c](0.8,2){$a$}
\rput[c](3.2,2){$b$}
\rput[c](2,3.3){$1$}
\rput[c](2,0.1){(a) $\alg{D}$}
\psdots(7,1)(6,2)(7,3)(8,2)
\psline(7,1)(6,2)(7,3)(8,2)(7,1)
\psline[linecolor=gray,arrows=<->](7,1.15)(7,2.85)
\psarc[linecolor=gray,arrows=->](5.7,2){.3}{30}{330}
\psarc[linecolor=gray,arrows=<-](8.3,2){.3}{210}{510}
\rput[c](7,.7){$a$}
\rput[c](5.8,2){$0$}
\rput[c](8.2,2){$1$}
\rput[c](7,3.3){$b$}
\rput[c](7,0.1){(b) $\spa{D}$}
\end{pspicture}
\caption{The De Morgan algebra $\alg{D}$ and the space $\spa{D}$}\label{Fig:Demorgan}
\end{center}
\end{figure}

\begin{theorem}[{\cite[Theorem~4.3.16]{CD98}}] \label{Th:DM_Duality}
 The structure $\spa{D}$ determines a strong natural duality between the variety of De~Morgan algebras and 
 the category of De Morgan spaces. 
 \end{theorem}

\noindent
An application of Theorem~\ref{Th:FreeDual} then establishes that for each $n \in \mathbb{N}$, the structure 
$\ADM(\spa{D}^n)$ is the free $n$-generated De~Morgan algebra.

The following lemma provides examples of (isomorphic copies of) subalgebras of $\F_{\DM}$ that will play an important role in the remainder of this section. 

\begin{lemma}\label{Lem_demorgansub}
Let $X$ be a finite set and $f$ an involution on $X$ with at least one fixpoint ($x \in X$ with $f(x) = x$). 
Consider $Y=\{u,v\}\cup X$ where $u \not = v$ and $u,v\notin X$, the map $\bar{f}\colon Y\to Y$ extending $f$ 
with  $f(u)=v$ and $f(v)=u$, and the partial order $\le$ on $Y$ defined by 
$$
x\leq y \qquad \Leftrightarrow \qquad x=u \quad {\rm or} \quad  y=v \quad {\rm or} \quad x=y.
$$
Then $\ADM(Y,\leq, \bar{f},\mathcal{P}(Y)) \in \cop{IS}(\F_{\DM})$.
\end{lemma}
\begin{proof}
Let $\{x_0,x_1,\ldots,x_n\}$  be an enumeration of the elements of $X$ such that $f(x_0)=x_0$.
For each $i \in\{1,\ldots,n\}$, let $e_i=(\delta_{i1},\ldots,\delta_{in})\in \spa{D}^{n}$ where $\delta_{ij}$ denotes the Kronecker delta; 
that is,  $\delta_{ij}$ is $1$ if $i=j$ and $0$ otherwise.

Now we define $\eta\colon\spa{D}^{n+2}\to Y$ as follows:
$$
\eta(z_1,\ldots,z_{n+2})=
\begin{cases}
v	& {\rm if } \ |\{i\mid z_i=a\}|<|\{i\mid z_i=b\}|\\
u	& {\rm if } \ |\{i\mid z_i=a\}|>|\{i\mid z_i=b\}|\\
x_i	& {\rm if } \ (z_1,\ldots,z_{n+2})=(e_i,a,b)\\
\bar{f}(x_i)		& {\rm if } \ (z_1,\ldots,z_{n+2})=(e_i,b,a)\\
x_0 & {\rm otherwise.}
\end{cases}
$$ 
It is a tedious but straightforward calculation to check that $\eta$ is a well-defined monotone map that preserves the involution. 

Since $\eta$ is onto, by Theorem~\ref{Th:DM_Duality}, $\ADM(\eta)\colon\ADM( Y,\leq, \bar{f},\mathcal{P}(Y))\to \ADM(\spa{D}^{n})$ 
is a one-to-one homomorphism. The result then follows by Theorem~\ref{Th:FreeDual}. 
\end{proof}

\noindent
We remark that the algebra $\ADM( Y,\leq, \bar{f},\mathcal{P}(Y))$ defined in this lemma not only embeds into $\F_{\DM}$, it 
is also a retract of $\F_{\DM}$. A characterization of finite retracts of free De~Morgan algebras, the finite projective algebras of this variety, 
is given in~\cite{BC201X}.


\subsection{Admissible Clauses in De~Morgan Algebras}

Our goal now is to describe the finite members of $\cop{IS}(\F_{\DM})$ via a characterization of their 
counterparts in the dual space (Lemma~\ref{Lem_ISPUdemorgan}), thereby obtaining a basis for the admissible clauses of 
$\cls{DMA}$ (Theorem~\ref{t:BasisDeMorgan}). 
 We begin by proving two preparatory lemmas.

\begin{lemma}\label{Lem:idempotent}
Let $\alg{A}$ be a finite De Morgan algebra and let $\XDM(\alg{A})=(X,\leq,f,\tau)$ be  its dual space. 
Then the following are equivalent:
\begin{itemize}
 \item[{\rm (i)}]  $\alg{A}$ satisfies \eqref{Eq:DMnontrivial}.
 \item[{\rm (ii)}] There exists $x\in X$ such that $f(x)=x$.
\end{itemize}
\end{lemma}
\begin{proof}
(i) $\Rightarrow$ (ii). Suppose that  $\alg{A}$ satisfies \eqref{Eq:DMnontrivial}. Then $\alg{A}$ is non-trivial and 
$\cls{DMA}[\alg{A},\alg{D}]=X\neq\emptyset$. 
Assume now for a contradiction that $x\neq f(x)$ for each $x\in X$. We show 
that $\alg{A}$ must contain an element $c$ such that $\lnot c = c$,  contradicting \eqref{Eq:DMnontrivial}. 
Let us define for each $x\in X$:
\[
\begin{array}{rcl}
C(x) & = & \{C\subseteq X \mid C\mbox{ is a chain and }\forall y\in C, y\leq x \}\\[.05in]
m(x) & = & \max\{|C|\mid C\in C(x) \}.
\end{array}
\]
We consider the sets:
\[
\begin{array}{rcl}
U	& = & \{x\in X\mid m(x)<m(f(x))\}\\[.05in]
V	& = & \{x\in X\mid m(f(x))<m(x)\}\\[.05in]
W	& = & \{x\in X\mid m(x)=m(f(x))\}.
\end{array}
\]
Clearly $U$, $V$, and $W$ are pairwise disjoint and $X =U\cup V\cup W$. Moreover, $f(U)=V$, $f(V)=U$, and $f(W)=W$. 
Observe also that $U$ is a decreasing set. If $x\in U$ and $y\leq x$, then $f(x)\leq f(y)$, $C(y)\subseteq C(x)$, and $C(f(x))\subseteq C(f(y))$. 
Thus $m(y)\leq m(x)<m(f(x))\leq m(f(y))$ and $y \in U$. Also, if $x\in W$ and $y<x$, then $m(y)<m(x)=m(f(x))<m(f(y))$, so $y\in U$. 
A similar argument  establishes that $V$ is an increasing set and that $W$ is an anti-chain.

Now let $W_1, W_2\subseteq W$  be a partition of $W$ such that $f(W_1)=W_2$. The existence of such a partition is ensured by the 
fact that $f(x)\neq x$ for each $x \in X$. 
Consider the map $\eta\colon X \to \spa{D}$ defined by:
$$
\eta(x)=
\begin{cases}
  a & {\rm if} \ x\in U\cup W_1\\
  b & {\rm if} \ x\in V\cup W_2.
\end{cases}
$$
It follows that $\eta$ is well-defined and that $\neg\eta(x)=\eta(x)$ for each $x\in X$. Moreover, $\eta$ is increasing and preserves $f$. 
So $\eta\in\ADM(\XDM(\alg{A}))$. Let $c=e_{\alg{A}}^{-1}(\eta)$, so that $\neg c=c$.

(ii) $\Rightarrow$ (i). 
Suppose that there exists $h\in X=\DM[\alg{A},\alg{D}]$ such that $f(h)=h$. Now assume for a contradiction that 
there exists a homomorphism $g\colon \alg{Fm}_{\lang}\to \alg{A}$ satisfying $g(x)=g(\neg x)$. 
On one hand, $h(g(x))=f(h)(g(x))=f(h(g(x)))$. 
That is, $h(g(x))$ is a fixpoint of $f$ in $\spa{D}$. Therefore, $h(g(x))\in\{0,1\}\subseteq \spa{D}$. On the other hand, 
$h(g(x))=h(g(\neg x))=\neg(h(g(x)))$ implies that $h(g(x))$ is a fixpoint of $\neg$ in $\alg{D}$. 
So $h(g(x))\in\{a,b\}\subseteq \spa{D}$, a contradiction.
Hence such a $g$ does not exist and  \eqref{Eq:DMnontrivial} is satisfied. 
\end{proof}

\begin{lemma}\label{Lem_ISPUdemorgan}
 Let $\alg{A}$ be a finite De Morgan algebra  and let $\XDM(\alg{A})=(X,\leq,f,\tau)$ be its dual space. Then the following are equivalent:
\begin{itemize}
 \item[{\rm (i)}] $\alg{A}\in\cop{IS}(\F_{\DM})$.
 \item[{\rm (ii)}]  $( X,\leq)$ is non-empty and bounded and there exists $x\in X$ such that $f(x)=x$.
\item[{\rm (iii)}] $\alg{A}$ satisfies \eqref{Eq:1prime} and \eqref{Eq:DMnontrivial}.
 \end{itemize}
\end{lemma}
\begin{proof}
(i) $\Rightarrow$ (ii). 
Suppose that  $\alg{A}\in\cop{IS}(\F_{\DM})$. Then, 
since $\alg{A}$ is finite, there exists $n\in\{1,2,\ldots\}$ and a one-to-one homomorphism 
$h\colon \alg{A}\to \F_{\DM}(n)$. So the map $\XDM(h)\colon \XDM( \F_{\DM}(n))\to \XDM(\alg{A})$ is 
onto. Since $\XDM(h)$ is order-preserving and $\XDM( \F_{\DM}(n))\cong\spa{D}^{n}$, it follows that $\XDM(\alg{A})$ is bounded. Similarly, 
since  $\XDM(h)$ preserves the involution and $\XDM(\F_{\DM}(n))$ has 
$2^n$ fixpoints for the involution (every element in $\{0,1\}^n$ is a fixpoint in  $\spa{D}^{n}$),  $\XDM(\alg{A})$ has at least one fixpoint.

 (ii) $\Rightarrow$ (i). Let us denote the top and bottom maps in $\XDM(\alg{A})$ by $v$ and $u$, respectively. 
   Let  $\leq'$ denote the partial order on $X$ defined as follows:
$$
x\leq' y\qquad \mbox{ iff }\qquad x=u \quad {\rm or} \quad y=v \quad {\rm or} \quad x=y.
$$
It follows from Lemma~\ref{Lem_demorgansub} that $\ADM(X, \leq', f,\mathcal{P}(X))\in\cop{IS}(\F_{\cls{DMA}})$.

Since $\leq'\subseteq \leq$, the identity map $\id_X$ on $X$  is an order and involution-preserving map from  
$( X, \leq', f,\tau)$ onto $\XDM(\alg{A})$. Hence, the map $\ADM(\id_X)\circ e_{\alg{A}}\colon \alg{A}\to 
\ADM( X, \leq', f,\tau)$ is one-to-one. So $\alg{A}\in \cop{IS}(\F_{\cls{DMA}})$. 

(ii) $\Rightarrow$ (iii). By Lemma~\ref{Lem:idempotent}, $\alg{A}$ satisfies \eqref{Eq:DMnontrivial}.
 To see that $\alg{A}$ satisfies \eqref{Eq:1prime}, let $x$ be the bottom element  and $f(x)$ the top element of $(X,\leq)$. 
Assume that $c,d\in\alg{A}$ are such that $c \vee d=\top$. Then the maps  $e_{\alg{A}}(c)$ and $e_{\alg{A}}(d)$ satisfy 
$e_{\alg{A}}(c)(y) \vee e_{\alg{A}}(d)(y)=1$ 
for each $y\in X$. In particular, 
$e_{\alg{A}}(c)(x)\vee e_{\alg{A}}(d)(x)=1$.  We claim that $e_{\alg{A}}(c)(x)=1$ or $e_{\alg{A}}(d)(x)=1$. 
Suppose for a contradiction that $e_{\alg{A}}(c)(x)\neq 1\neq e_{\alg{A}}(d)(x)$. Then 
only two cases are possible: $e_{\alg{A}}(c)(x)=a$ and $e_{\alg{A}}(d)(x)=b$, or $e_{\alg{A}}(d)(x)=a$ and $e_{\alg{A}}(c)(x)=b$. If  $e_{\alg{A}}(c)(x)=b$, then $e_{\alg{A}}(c)(f(x))=a$, but since $x$ is the bottom element of $( X,\leq)$, it follows that 
$x\leq f(x)$ and  $e_{\alg{A}}(c)(x)\nleq e_{\alg{A}}(c)(f(x))$ which contradicts the monotonicity of 
$e_{\alg{A}}$. A similar contradiction is obtained if we assume $e_{\alg{A}}(d)(x)=b$. 
Hence, without loss of generality, assume $e_{\alg{A}}(c)(x)=1$. Then $e_{\alg{A}}(c)(f(x))=1$. Since $e_{\alg{A}}(c)$ is a monotone map, for each $y\in X$, 
$1=e_{\alg{A}}(c)(x)\leq e_{\alg{A}}(c)(y)\leq e_{\alg{A}}(c)(f(x))=1$. That is, $e_{\alg{A}}(c)$ is constantly $1$, establishing $c=\top$ as required. 

(iii) $\Rightarrow$ (ii). If $\alg{A}$ satisfies  \eqref{Eq:1prime} and \eqref{Eq:DMnontrivial}, then $\top$ is  join irreducible and  $\bot$ is meet 
irreducible and $\bot\neq\top$  in $\alg{A}$. Hence, the map $h\colon \alg{A}\to\alg{D}$ defined by $h(\top)=1$, $h(\bot)=0$, and $h(x)=a$ if $x\notin\{\bot,\top\}$ is 
 a homomorphism and therefore the bottom element of $\XDM(\alg{A})$. Clearly, $f(h)$ is the top element of $(X,\leq)$. 
 So  $( X,\leq)$ is non-empty and bounded. Also, by Lemma~\ref{Lem:idempotent}, there exists $x\in X$ such that $f(x)=x$.

\end{proof}

Now immediately, by Lemmas~\ref{l:mainlemma} and~\ref{Lem_ISPUdemorgan}:

\begin{theorem}\label{t:BasisDeMorgan}
$\{(\ref{Eq:1prime}),(\ref{Eq:DMnontrivial})\}$ is a basis for the admissible clauses of De~Morgan algebras.
\end{theorem}


\subsection{Admissible Quasi-identities in De~Morgan Algebras}

In order to present a basis for the admissible quasi-identities of De~Morgan algebras 
we again need some preparatory lemmas.

\begin{lemma}\label{Lem_fixpointovermin}
Let $\alg{A}$  be a finite De Morgan algebra and let $\XDM(\alg{A})=(X,\leq,f,\tau)$ be its dual space. Then the following are equivalent:
\begin{itemize}
\item[{\rm (i)}] $\alg{A}$ satisfies \eqref{Eq_DeM1}.
\item[{\rm (ii)}] For every $x\in \min(X,\leq)$, there exists $z\in X$ such that $x\leq z$ and $f(z)=z$.
\end{itemize}
\end{lemma}
\begin{proof}
(i) $\Rightarrow$ (ii). If $\alg{A}$ is trivial, then (ii) is clearly satisfied.
Suppose then for a contradiction that $\alg{A}$ is non-trivial and that there exists $x\in \min(X,\leq)$ such that for all $z\geq x$, 
$f(z)\neq z$. 

Let $Y=\{z\in X\mid f(z)\neq z\}$. Then $Y$ with the structure inherited from  $\XDM(\alg{A})$ is a non-trivial De~Morgan space without fixpoints. 
In the proof of  Lemma~\ref{Lem:idempotent}, it was shown that given such a De Morgan space $Y$ without fixpoints, there exists a map 
$\eta\colon Y\to \spa{D}$ such that $\neg \eta=\eta$; that is, $\eta(Y)\subseteq \{a,b\}$.

Now, consider $\gamma,\mu,\nu\colon \XDM(\alg{A}) \to \spa{D}$ defined as follows:
$$
\gamma(u)=
\begin{cases}
\eta(u)  	& {\rm if }\ u\neq f(u)\\
0	& {\rm if }\ u=f(u)
\end{cases}
\qquad
\mu(u)=
\begin{cases}
0	& {\rm if} \ x\leq u\leq f(x)\\
b	& {\rm if} \ x< u \ {\rm and } \ u\nleq f(x)\\
a	& {\rm if} \ u< f(x) \ {\rm and }\ x\nleq u\\
1	& {\rm otherwise}
\end{cases}
\qquad
\nu(u)=
\begin{cases}
a	& {\rm if } \ x= u\\
b	& {\rm if } \ f(x)= u\\
1	& {\rm otherwise.}
\end{cases}
$$
It is an easy calculation to prove that $\gamma,\mu,\nu\in\ADM(\XDM(\alg{A}))$. Now observe that since $\gamma(X)\subseteq\{0,a,b\}$ and  
$(\gamma\vee\mu)(X)\in\{a,b,1\}$, it follows that $\gamma\leq\neg\gamma$ and $\neg(\gamma\vee \mu)\leq \gamma\vee\mu$. Moreover, 
since $\nu(u)\neq 1$ iff $u\in\{x,f(x)\}$ and $\mu(\{x,f(x)\})=\{0\}$, $\neg\mu\vee\nu=1$ and  $\nu\neq 1$, contradicting \eqref{Eq_DeM1}.

(ii) $\Rightarrow$ (i). 
Let $c,d,e\in A$ be such that $c\leq\neg c$, $\neg(c\vee d)\leq c\vee d$, and $\neg d\vee e=\top$. 
Consider $x\in \min(X,\leq)$. By hypothesis, there exists $y\in X$ such that $x\leq y$ and $f(y)=y$.
Then $e_{\alg{A}}(c)(y)=e_{\alg{A}}(c)(f(y))$ and $e_{\alg{A}}(d)(y)=e_{\alg{A}}(d)(f(y))$. So 
$e_{\alg{A}}(c)(y),e_{\alg{A}}(d)(y)\in\{0,1\}$. Since $e_{\alg{A}}(c)(y)\leq \neg(e_{\alg{A}}(c)(y))$, it follows that 
$e_{\alg{A}}(c)(y)\in\{0,a,b\}$ and therefore $e_{\alg{A}}(c)(y)=0$.

 Thus $\neg(e_{\alg{A}}(c)(y)\vee e_{\alg{A}}(d)(y))=\neg(e_{\alg{A}}(d)(y))\leq e_{\alg{A}}(c)(y)\vee e_{\alg{A}}(d)(y)=e_{\alg{A}}(d)(y)$, 
which implies that $e_{\alg{A}}(d)(y)=1$ and therefore $\neg e_{\alg{A}}(d)(y)=0$. 
Then  $\neg e_{\alg{A}}(d)(x)\in\{0,a\}$ and $e_{\alg{A}}(e)(y)=e_{\alg{A}}(e)(f(y))$. Therefore,  $e_{\alg{A}}(e)(y)=1$. 
On one hand, since $e_{\alg{A}}(e)(x)\leq e_{\alg{A}}(e)(y)=1$, it follows that $e_{\alg{A}}(e)(x)\in \{a,1\}$. On the other hand, since 
$e_{\alg{A}}(e)(x)\vee  \neg e_{\alg{A}}(d)(x)=1\in \alg{D}$, it follows that 
$e_{\alg{A}}(e)(x)\in \{b,1\}$. So $e_{\alg{A}}(e)(x)=1$. 
Since this can be proved for every minimal element of 
$(X,\leq)$ and every maximal element is equal to $f(x)$ for some minimal element,  $e_{\alg{A}}(e)(x)=1$ 
for every minimal or maximal element $x$ of $(X,\leq)$. Since $e_{\alg{A}}(e)$ is increasing, we conclude that 
$e_{\alg{A}}(c)$ is constantly equal to $1$; that is, $c=\top$. 
\end{proof}

\begin{lemma}\label{Lem_coverinterval}
Let $\alg{A}$  be a finite De Morgan algebra and let $\XDM(\alg{A})=(X,\leq,f,\tau)$ be its dual space. 
Then the following are equivalent:
\begin{itemize}
\item[{\rm (i)}] $\alg{A}$ satisfies \eqref{Eq_DeM2}.
\item[{\rm (ii)}] For every $x\in X$, there exists $y\in X$ such that $y\leq x,f(x)$.
\end{itemize}
\end{lemma}
\begin{proof}
(i) $\Rightarrow$ (ii). 
If $\alg{A}$ is trivial, then (ii) is clearly satisfied. Otherwise, suppose that $\alg{A}$ is non-trivial and  that there exists 
$x\in \XDM(\alg{A})$ such that for all $y\in\XDM(\alg{A})$, either 
$y\nleq x$ or  $y\nleq f(x)$. Consider $\eta,\mu\colon \XDM(\alg{A}) \to \spa{D}$ defined as follows:
$$
\eta(u)=
\begin{cases}
a	& {\rm if } \ u\leq x\\
b	& {\rm if } \ u\geq f(x)\\
0	& {\rm otherwise}
\end{cases}
\qquad
\mu(u)=
\begin{cases}
a	& {\rm if} \ u\leq f(x)\\
b	& {\rm if} \ u\geq x\\
0	& {\rm otherwise.}
\end{cases}
$$
If $u\leq x$ and $u\geq f(x)$, then $f(x)\leq x$, contradicting the assumption that, since $f(x) \le f(x)$, 
$f(x) \nleq x$. So $\eta$ is well-defined. Similarly, $\mu$ is well-defined.

It is easily observed that $\eta,\mu\in\ADM(\XDM(\alg{A}))$ and that 
$\eta\leq\neg\eta$, $\mu\leq \neg\mu$. Letting $c=e_{\alg{A}}^{-1}(\eta)$ and $d=e_{\alg{A}}^{-1}(\mu)$, 
it follows that $c\leq\neg c$ and $d\leq \neg d$.  
Consider $u\in \XDM(\alg{A})$. If $\eta(u)=a$, then $u\leq x$ and, by assumption, 
$u\nleq f(x)$. Thus $\mu(u)\in\{b,0\}$, which implies $(\eta\wedge\mu)(u)=\eta(u)\wedge\mu(u)=0\in\alg{D}$. 
Similarly, if $\eta(u)=b$, then $u\geq f(x)$ and $u\nleq x$. Thus $\mu(u)\in\{a,0\}$, and $(\eta\wedge\mu)(u)=0$. 
So $\eta\wedge \mu$ is constantly $0$; that is, $c\wedge d=\bot$.  
But $(\eta\vee\mu)(x)=\eta(x)\vee\mu(x)=1$, and therefore $(\eta\vee\mu)(x)=1\nleq\neg(\eta\vee\mu)(x)=0$; 
that is, $c\vee d\nleq \neg(c\vee d)$ which contradicts \eqref{Eq_DeM2}.

(ii) $\Rightarrow$ (i). 
Consider $c,d\in A$ such that 
$c\leq\neg c$, $d\leq\neg d$, and $c\wedge d=\bot$. 
Let $x\in\XDM(\alg{A})$. Suppose now for a contradiction that $(e_{\alg{A}}(c)\vee e_{\alg{A}}(d))(x)=1$. 
It follows from the assumptions that either $e_{\alg{A}}(c)(x)=a$ and $e_{\alg{A}}(d)(x)=b$, 
or $e_{\alg{A}}(c)(x)=b$ and $e_{\alg{A}}(d)(x)=a$. If $e_{\alg{A}}(c)(x)=a$ and $e_{\alg{A}}(d)(x)=b$, then 
 there exists $y\in\XDM(\alg{A})$ such that $y\leq x,f(x)$. But then $e_{\alg{A}}(c)(y)\leq e_{\alg{A}}(c)(x)=a$ and 
 $e_{\alg{A}}(d)(y)\leq e_{\alg{A}}(d)(f(x))=a$. So $e_{\alg{A}}(c)(y)\wedge e_{\alg{A}}(d)(y)=a\neq 0$, contradicting the 
 assumption that $c\wedge d=\bot$. A similar contradiction is obtained if we assume $e_{\alg{A}}(c)(x)=b$ and 
 $e_{\alg{A}}(d)(x)=a$.

We conclude that $(e_{\alg{A}}(c)\vee e_{\alg{A}}(d))(x)\neq 1$ for each $x\in\XDM(\alg{A})$; 
that is $e_{\alg{A}}(c\vee d)(\XDM(\alg{A}))\subseteq \{0,a,b\}$. Therefore $c\vee d\leq\neg(c\vee d)$. 
\end{proof}

\begin{lemma}\label{l:dmaisp}
Let $\alg{A}$  be a finite non-trivial De~Morgan algebra and let $\XDM(\alg{A})=(X,\leq,f,\tau)$ be its dual space. 
Then the following are equivalent:
\begin{itemize}
  \item[{\rm (i)}] $\alg{A}\in\cop{ISP}(\F_{\cls{DMA}})$.
  \item[{\rm (ii)}] $(X,\leq,f,\tau)$ satisfies the following conditions:
  \begin{itemize}
    \item[{\rm (a)}] For every $x\in \min(X,\leq)$, there exists $z\in X$ such that $x\leq z=f(z)$.
    \item[{\rm (b)}] For every $x\in X$, there exists $y\in X$ such that $y\leq x,f(x)$.
  \end{itemize}
  \item[{\rm (iii)}] $\alg{A}$ satisfies \eqref{Eq_DeM1} and \eqref{Eq_DeM2}.
\end{itemize}
\end{lemma}
\begin{proof}
(ii) $\Leftrightarrow$ (iii). Immediate from  Lemmas~\ref{Lem_fixpointovermin} and~\ref{Lem_coverinterval}.

(i) $\Rightarrow$ (ii). Since $\alg{A}$ is finite, by (i), there exist $n,m\in\{1,2,\ldots\}$ and a one-to-one homomorphism 
$h\colon\alg{A}\to\F_{\cls{DMA}}(n)^m$.
So $\XDM(h)\colon \coprod_{i=1}^m {\spa{D}}^n\to \XDM(\alg{A})$ is an onto map.

Let $x\in X$ and $y\in  \coprod_{i=1}^m \spa{D}^n$ be such that $\XDM(h)(y)=x$.
Note that $\coprod_{i=1}^m  \spa{D}^n$ is just the  disjoint union of $m$ different copies $\spa{D}^n _i$ of $\spa{D}^n$. So 
there is an $i$ such that $y\in \spa{D}^n_i$. Let $z=(a,\ldots,a)\in  \spa{D}^n_i$, then $\XDM(h)(z)\leq x,f(x)$ which proves (b).

If $x\in  \min(X,\leq)$, then  $\XDM(h)(z)=x$. Let $t=(0,\dots,0)\in D^n_i$, then $x\leq\XDM(h)(t)=f(\XDM(h)(t))$, and (a) follows.

(ii) $\Rightarrow$ (i). For each $x\in \min(X,\leq)$, let $[x,f(x)]=\{z\in X\mid x\leq z\leq f(x)\}$. By condition (a), 
there exists $z\in[x,f(x)]$ such that $f(z)=z$.  Let us consider $[x,f(x)]$  endowed with the involution, poset, and 
topological structure inherited from $X$. Then $[x,f(x)]$ satisfies all the conditions of Lemma~\ref{Lem_ISPUdemorgan} (iii). 
Hence
\[
\ADM([x,f(x)])\in\cop{IS}(\F_{\DM}).
\]
By definition, trivially the inclusion map $\iota_x\colon [x,f(x)]\to X$ 
is monotone and preserves the involution. Consider $Y=\coprod_{x\in\min(X,\leq)} [x,f(x)]$ and  $\iota\colon Y\to X$ 
the unique morphism such that $\iota|_{[x,f(x)]}=\iota_x$. By (b), $\iota$ is an onto morphism in the category of Priestley spaces 
with an involution. Now Theorem~\ref{Th:DM_Duality} implies that the homomorphism 
$\ADM(\iota)\colon \ADM(\XDM(\alg{A}))\to \ADM(Y)$ is one-to-one. However,
$$
\textstyle\ADM(Y)=\ADM(\coprod_{x\in\min(X,\leq)} [x,f(x)])\cong\prod_{x\in\min(X,\leq)} \ADM([x,f(x)])\in \cop{ISP}(\F_{\DM}).
$$ 
Hence $\alg{A}\cong\ADM(\XDM(\alg{A}))\in \cop{ISP}(\F_{\DM})$ as required. 
\end{proof}

Again, immediately by Lemmas~\ref{l:mainlemma}~and~\ref{l:dmaisp}:

\begin{theorem}\label{t:BasisDeMorganQ}
$\{(\ref{Eq_DeM1}),(\ref{Eq_DeM2})\}$ is a basis for the admissible quasi-identities of De~Morgan algebras.
 \end{theorem}
 

\subsection{Admissibility in De~Morgan Lattices} \label{Sec:DemLat}

Characterizations of admissibility in De~Morgan lattices follow relatively straightforwardly from the above characterizations 
for De~Morgan algebras. We first extend the definition of $\overline{\alg{L}}$ 
given in Section~\ref{s:distributive} to the case of De~Morgan algebras and lattices. 
Let $\alg{A}=(A,\wedge,\vee,\neg)$ be a De~Morgan lattice. Then $\overline{\alg{A}}=
(A,\wedge,\vee,\neg',\top,\bot)$ denotes the 
De~Morgan algebra that results from adding to the lattice $(A,\wedge,\vee)$ fresh top and bottom elements $\top$ and $\bot$ with  
$\neg'(\top)=\bot$, $\neg'(\bot)=\top$, and  $\neg'(x)=\neg(x)$ otherwise. Note that $\overline{{\alg{A}}}$ 
satisfies \eqref{Eq:1prime}.

\begin{lemma}\label{l:dml}
Let $\alg{A}$  be a finite De~Morgan lattice. Then the following are equivalent:
\begin{itemize}
\item[{\rm (i)}] $\alg{A}\in\cop{IS}(\F_{\cls{DML}})$. 
\item[{\rm (ii)}] $\alg{A}$ satisfies \eqref{Eq:DMnontrivial}.
\end{itemize}
\end{lemma}
\begin{proof}
 Observe first that $\alg{F}_{\DM}(\kappa)\cong \overline{\alg{F}}_{\cls{DML}}(\kappa)$. Hence, easily, 
  each finite  De~Morgan lattice $\alg{A}$ is in  
 $\cop{IS}(\F_{\cls{DML}})$ iff $\overline{\alg{A}}\in\cop{IS}(\F_{\DM})$. Moreover,  by Lemma~\ref{Lem_ISPUdemorgan},  
 $\overline{\alg{A}}\in\cop{IS}({\F_{\DM}})$ iff $\overline{\alg{A}}$ satisfies $\eqref{Eq:1prime}$ and $\eqref{Eq:DMnontrivial}$. Hence, 
 since $\overline{\alg{A}}$ always satisfies \eqref{Eq:1prime},  $\alg{A}\in\cop{IS}({\F_{\cls{DML}}})$ iff 
 $\overline{\alg{A}}$ satisfies 
$\eqref{Eq:DMnontrivial}$. Since $\overline{\alg{A}}$ satisfies 
$\eqref{Eq:DMnontrivial}$ iff  
 $\alg{A}$ satisfies $\eqref{Eq:DMnontrivial}$,  the result follows. 
\end{proof}
\begin{lemma}\label{l:dml2}
Let $\alg{A}$  be a finite De~Morgan lattice. Then the following are equivalent:
\begin{itemize}
\item[{\rm (i)}] $\alg{A}\in\cop{ISP}(\F_{\cls{DML}})$. 
\item[{\rm (ii)}] $\alg{A}$ satisfies \eqref{deMorganrule}.
\end{itemize}
\end{lemma}
\begin{proof}
The result follows from Lemma~\ref{l:dml} and the fact that a De Morgan lattice $\alg{B}$ satisfies $\eqref{deMorganrule}$ iff 
it is a trivial algebra or satisfies \eqref{Eq:DMnontrivial}. 
\end{proof}
So by Lemmas~\ref{l:mainlemma},~\ref{l:dml},~and~\ref{l:dml2}:

\begin{theorem}\label{t:DeMorganLattices}
$\{\eqref{Eq:DMnontrivial}\}$ and $\{(\ref{deMorganrule})\}$ are bases for the admissible clauses and 
admissible quasi-identities of De~Morgan lattices, respectively.
\end{theorem}


\section{Kleene Algebras and Lattices}\label{s:kleene}

De~Morgan algebras and lattices have only two non-trivial proper subvarieties apiece, Boolean algebras 
and lattices, and Kleene algebras and lattices (see~\cite{Kal58}). 
As shown in Example~\ref{e:Bool}, 
Boolean algebras and Boolean lattices are non-negative universally complete but not universally complete. 
However, the varieties of 
Kleene algebras and lattices both have admissible quasi-identities and clauses that are not derivable.

Recall that \emph{Kleene lattices} and {\em Kleene algebras} are, respectively, De~Morgan lattices 
and De~Morgan algebras satisfying the Kleene condition 
$a \wedge \neg a \le b \vee \neg b$ for all elements $a,b$.
Consider now the  clauses: 
\begin{align}
\nonumber \top\eq\bot 							&\quad\imp\quad\emptyset	\tag{\ref{Eq:nontrivial}}\\
\nonumber x\vee y \eq \top 						& \quad \imp \quad  x\eq \top ,\  y\eq \top \tag{\ref{Eq:1prime}}\\
\nonumber x\eq\lnot x							&  \quad\imp\quad\emptyset	\tag{\ref{Eq:DMnontrivial}}\\
 {\neg x \leqn x,\ x\wedge \neg y\leqn \neg x\vee y} 		& \quad \imp \quad \neg y \leqn y. \label{Kleenerule}
\end{align} 
In this section, we prove the following:

\begin{itemize}

\item The clauses \eqref{Eq:nontrivial}, \eqref{Eq:1prime}, and \eqref{Kleenerule} 
		form a basis for the admissible clauses of Kleene algebras (Theorem~\ref{t:KleeneClauses}).
		
\item The quasi-identity \eqref{Kleenerule} forms a basis for both the admissible quasi-identities of Kleene algebras 
		(Theorem~\ref{t:BasisKleeneQ}) and the admissible quasi-identities of Kleene lattices (Theorem~\ref{t:Kleenemain}). 
		
\item The clauses \eqref{Eq:DMnontrivial} and \eqref{Kleenerule} 
		form a basis for  the admissible clauses of Kleene lattices  (Theorem~\ref{t:Kleenemain}).

\end{itemize}

\noindent
In~\cite{MR12}, an alternative proof  is given of the fact that~(\ref{Kleenerule}) forms a basis for the 
admissible quasi-identities of Kleene lattices and algebras. This proof makes use of a description
 of the  finite lattice of quasivarieties of Kleene lattices provided by Pynko in~\cite{Pyn99}. 
 However, as in the case of De~Morgan algebras, a suitable natural duality provides a method for 
characterizing also the admissible clauses of these varieties. 


\subsection{A Natural Duality for Kleene Algebras}

Let $\alg{K}=(\{0,a,1\},\wedge,\vee,\neg,0,1)$ be the Kleene algebra where $0 < a < 1$
and $\neg 0=1$, $\neg 1=0$, and $\neg a=a$ (see Figure~\ref{Fig:Kleene}(a)). Then $\cls{KA} = \Qg(\alg{K})$ (see~\cite{Kal58}). 
Let also $\spa{K}=( \{  0,a,1\} ,\leq_K,\sim_K,\{  0,1\},\mathcal{P}( \{  0,a,1\}))$ where $\leq_K$ is defined by $0,1\leq_K a$ and $0,1$ 
are incomparable (see Figure~\ref{Fig:Kleene}(b)), and $\sim_K\subseteq K^2 $ is as depicted in Figure~\ref{Fig:Kleene}(c).

\begin{figure}
\begin{center}
\begin{pspicture}(0,0)(11,6)
\psdots(1,1)(1,2)(1,3)
\psline(1,1)(1,3)
\pscurve[linecolor=gray,arrows=<->](1.1,1.1)(1.5,2)(1.1,2.9)
\psarc[linecolor=gray,arrows=->](1.2,2){.2}{230}{500}
\rput[c](1,0.7){$0$}
\rput[c](1,3.3){$1$}
\rput[c](0.75,2.05){$a$}
\rput[c](1,0){(a) $\alg{K}$}

\psdots(4,1)(5,2)(6,1)
\psline(4,1)(5,2)(6,1)
\rput[c](4,0.7){$0$}
\rput[c](6,0.7){$1$}
\rput[c](5,2.3){$a$}
\rput[c](5,0){(b) $(K,\leq_K)$}

\psdots(9,1)(10,2)(9,3)(8,4)(9,5)(10,4)(8,2)
\psline(9,1)(10,2)(9,3)(8,4)(9,5)(10,4)(8,2)(9,1)
\rput[c](9,0.7){$(0,0)$}
\rput[c](9,3.55){$(a,a)$}
\rput[c](9,5.3){$(1,1)$}
\rput[c](7.7,4.3){$(a,1)$} 
\rput[c](10.3,4.3){$(1,a)$}
\rput[c](7.7,2.3){$(0,a)$} 
\rput[c](10.3,2.3){$(a,0)$}
\rput[c](9,0){(c) $\sim_K$}

\end{pspicture}
\caption{The Kleene algebra $\alg{K}$ and the Kleene  space $\spa{K}$}\label{Fig:Kleene}
\end{center}
\end{figure}
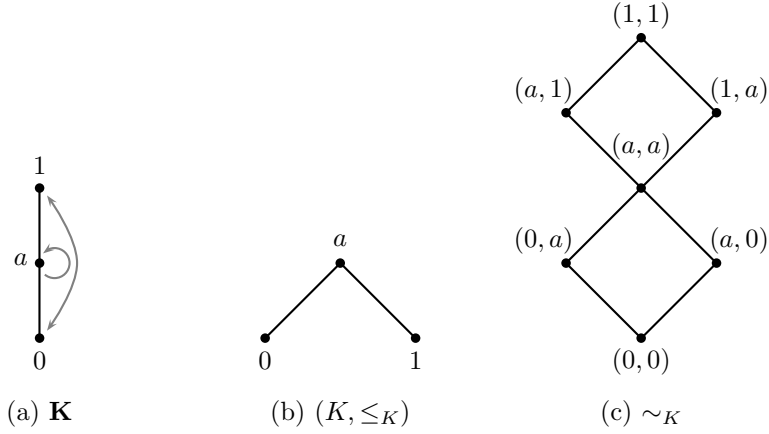
A structure $( X,\leq_X,\sim_X,Y_X, \tau_X) $ is called a {\em Kleene space} if
\begin{itemize}
\item $( X,\leq_X, \tau_X)$ is a Priestley space.
\item $\sim_X$ is a closed binary relation; i.e., $\sim_X$ is a closed subset of $X^{2}$.
\item $Y_X$ is a closed subset of $X$.
\item For all $x,y,z\in X$:
\begin{itemize}
\item $x\sim_Xx$.
\item	If $x\sim_Xy$ and $x\in Y_X$, then $x\leq_X y$.
\item If $x\sim_X y$ and $y\leq_X z$, then $z\sim_X x$.
\end{itemize}
\end{itemize}

\begin{theorem}[\cite{DW85}] \label{Theo:KleeneDuality} $\spa{K}$ determines a strong natural 
duality between the variety of Kleene algebras and the category of Kleene spaces (see also~\cite[Theorem~4.3.10]{CD98}).
\end{theorem}

The following lemma provides a subclass of $\cop{IS}(\F_{\KA})$ that will play a crucial role in the proof of Lemma~\ref{Lem_ISPUkleene}.

\begin{lemma}\label{Lem_kleeneSub}
 Let $( X,\leq)$ be a finite tree; that is, $( X,\leq)$ is a poset with a top element such that for each $x,y\in X$, $z\leq x,y$ 
 for some $z\in X$ iff ($x\leq y$ or $y\leq x$). Let  $ \sim_X=\leq\cup\geq$ and $Y_X=\min(X,\leq) $. 
 Then $\AK(X,\leq,\sim_X,Y_X,\mathcal{P}(X)) \in \cop{IS}(\F_{\KA})$.
\end{lemma}
\begin{proof}
Observe first that for each $x,y\in \spa{K}^{n}$, 
$x\sim_{\mbox{\footnotesize{$\spa{K}^n$}}}\!y$ iff there exists $z\leq x,y$. On the other hand, from the definition of Kleene spaces, 
it follows that for each Kleene space $(Z,\leq_Z,\sim_Z,Y_Z,\tau_Z)$, if $x,y,z\in Z$  are such that $z\leq_Z x,y$, then $x\sim_Z y$. 
This proves that any monotone map from $\spa{K}^{n}$ 
into a Kleene space $(Z,\leq_Z,\sim_Z,Y_Z,\tau_Z)$ is a morphism in the category of Kleene spaces iff it sends $\{0,1\}^{n}$ into $Y_Z$.

We show that there exist $n\in\{1,2,\ldots\}$ and a monotone onto map $\eta\colon\spa{K}^n\to (X,\leq,\sim_X,Y_X,\mathcal{P}(X))$ 
such that $\eta(\{0,1\}^{n})=Y_X=\min(X,\leq)$.

For $n=1,2,\ldots$, let $(BT(n),\leq_n)$ denote the perfect binary tree of depth $n$ (i.e., a tree in which every non-leaf node has 
exactly $2$ children and all leaves are at depth $n$). Then for some $n \in \{1,2,\ldots\}$, there is 
  an onto monotone map  $\eta_1\colon BT(n)\to (X,\leq)$ 
such that $\eta_1(\min(BT(n),\le_n))=\min(X,\leq)$. Consider $\sim_{BT(n)}=\leq_n\cup\geq_n$ and $Y_{BT(n)}=\min(BT(n),\leq_n)$. 
Then   $(BT(n),\leq_n,\sim_{BT(n)},Y_{BT(n)} ,\mathcal{P}(BT(n)))$ is a Kleene space. 
It follows that $\eta_1$ is a morphism of Kleene spaces from  $(BT(n),\leq_n,\sim_{BT(n)},Y_{BT(n)} ,\mathcal{P}(BT(n)))$ 
onto $(X,\leq,\sim_X,Y_X,\mathcal{P}(X))$.

Let $\eta_2\colon \spa{K}^{n}\to \spa{K}^{n}$ be defined as follows:
$$
\eta_2(x_1,\ldots,x_n)=
\begin{cases}
(a,\ldots,a)				& {\rm if }\ x_1=a\\
(x_1,\ldots,x_n)		& 	{\rm if}\ x_j \neq a \ {\rm for} \ j = 1 \ldots n\\		
(x_1,\ldots,x_{i-1},a,\dots,a)	& {\rm if }\ i=\min\{j\in\{1,\ldots,n\}\mid x_j=a\}.\\		
\end{cases}
$$
Then $\eta_2$ is monotone and acts as the identity on $\{0,1\}^n$, and is 
hence  a morphism of Kleene spaces.

By induction on $n$ it is easy to prove that $\eta_2(\spa{K}^n)$ is a perfect binary tree of depth $n$. Therefore, considering it as a 
subobject of $\spa{K}^{n}$, there is an isomorphism of Kleene spaces $\eta'$ 
from $\eta_2(\spa{K}^n)$ onto  $(BT(n),\leq_n,\sim_{BT(n)},Y_{BT(n)} ,\mathcal{P}(BT(n)))$. 

Finally, $\eta=\eta_1\circ \eta' \circ \eta_2$ is a morphism of Kleene spaces from $\spa{K}^n$ onto $(X,\leq,\sim_X,Y_X,\mathcal{P}(X))$. 
By Theorem~\ref{Theo:KleeneDuality}, $\AK(\eta)\colon \AK(X,\leq,\sim_X,Y_X,\mathcal{P}(X))\to \AK(\spa{K}^n)$ 
is a one-to-one homomorphism; that is, $\AK(X,\leq,\sim_X,Y_X,\mathcal{P}(X))\in\cop{IS}(\F_{\KA}(n))\subseteq \cop{IS}(\F_{\KA})$, 
as required. 
\end{proof}

The algebra $\AK(X,\leq,\sim_X,Y_X,\mathcal{P}(X))$ defined in this lemma not only embeds into $\F_{\KA}$, it 
is also a retract of $\F_{\KA}$. 
 A characterization of finite retracts of free Kleene algebras, the finite projective algebras of this variety, 
is given in~\cite{BC201X}.


\subsection{Admissibility in Kleene Algebras and Lattices}\label{Sec:KleLat}

We obtain bases for admissibility in Kleene algebras and lattices following a similar procedure to that 
described in Section~\ref{s:demorgan}. First, we describe two properties of the dual space of finite non-trivial 
Kleene algebras using clauses.

\begin{lemma}\label{Lem:KleneConds}
For a finite non-trivial Kleene algebra $\alg{A}$ with dual space $\XK(\alg{A})=(X,\leq,\sim,Y,\tau)$: 
 \begin{itemize}
\item[{\rm (a)}] $\alg{A}$ satisfies \eqref{Eq:nontrivial} iff   $X\neq\emptyset$.
 \item[{\rm (b)}]	 $\alg{A}$ satisfies \eqref{Eq:nontrivial} and \eqref{Eq:1prime} iff  $( X,\leq)$ has a top element.
 \item[{\rm (c)}]	$\alg{A}$ satisfies \eqref{Eq:nontrivial} and \eqref{Kleenerule} iff  $Y=\min( X,\leq)\neq\emptyset$.
\end{itemize}
\end{lemma}
\begin{proof}
(a) Straightforward, since $\alg{A}$ satisfies \eqref{Eq:nontrivial} iff it is non-trivial.

(b) Very similar to the  argument used in the proof of Lemma~\ref{Lem_ISPUdemorgan} (ii) $\Rightarrow$ (iii).

(c) Assume for a contradiction that $\alg{A}$ satisfies \eqref{Eq:nontrivial} and \eqref{Kleenerule} and $Y\subsetneq\min(X,\leq)$, and consider 
$y\in Y\setminus\min(X,\leq)$. We define
$\eta,\mu\colon \XK(\alg{A})\to \spa{K}$ as follows:
$$
\eta(x)=
\begin{cases}
 a & {\rm if} \ x\geq y\\
 1 & {\rm otherwise}
\end{cases}
\qquad
\mu(x)=
\begin{cases}
 a & {\rm if } \ x\geq y \ {\rm and } \ x\neq y\\
 0 & {\rm if } \ x=y\\
 1 & {\rm otherwise.}
\end{cases}
$$
Then $\eta,\mu\in \AK(\XK(\alg{A}))$, $\neg \eta\leq \eta$, $\eta\wedge \neg \mu=\neg \eta\leq \neg \eta\vee \mu$, and $\neg \mu \not\leq \mu$. 
Let $c=e_{\alg{A}}^{-1}(\eta)$ and $b=e_{\alg{A}}^{-1}(\mu)$. Then $\neg c\leq c$, $c\wedge \neg d\leq \neg c\vee d$, and 
$\neg d \not\leq d$, contradicting the assumption that $\alg{A}$ satisfies \eqref{Kleenerule}.

Conversely, assume that $Y=\min(X,\leq)\neq\emptyset$ and let $c,d\in A$ be such that $\neg c \leq c$ and $c\wedge \neg d\leq \neg c\vee d$. 
Then  $\neg c \leq c$ implies $e_{\alg{A}}(c)(x)\in\{a,1\}$ for each $x\in \XK(\alg{A})$.
Since $Y=\min(X,\leq)$, it follows that $e_{\alg{A}}(c)(y)=1$ 
for each $y\in \min(X,\leq)$.  If there exists $x\in\XK(\alg{A})$ such that $e_{\alg{A}}(d)(x)=0$,  then there exists $y\in\min(X,\leq)$ such that $e_{\alg{A}}(d)(y)=0$. 
It then follows that:
$$\begin{tabular}{rcl}
$ e_{\alg{A}}(c)(y)$&$=$ &$ e_{\alg{A}}(c)(y)\wedge 1$\\
&$ =$&$e_{\alg{A}}(c)(y)\wedge e_{\alg{A}}(\neg d)(y)$\\
&$\leq $&$e_{\alg{A}}(\neg c)(y)\vee e_{\alg{A}}(d)(y)$\\
&$=$&$e_{\alg{A}}(\neg c)(y)\vee 0$\\
&$=$&$e_{\alg{A}}(\neg c)(y).$
\end{tabular}
$$ 
 This contradicts  $e_{\alg{A}}(c)(y)=1$. So $e_{\alg{A}}(d)(x)\in\{a,1\}$  for each 
$x\in\XK(\alg{A})$,   which implies that $\neg e_{\alg{A}}(d)(x)\leq e_{\alg{A}}(d)(x)$ and hence $\neg d\leq d$. 
\end{proof}

\begin{lemma}\label{Lem_ISPUkleene}
  Let $\alg{A}$ be a finite 
 Kleene algebra and let $\XK(\alg{A})=(X,\leq,\sim,Y,\tau)$ be its dual space. 
  Then the following are equivalent:
\begin{itemize}
 \item[{\rm (i)}] $\alg{A}\in\cop{IS}(\F_{\KA})$.
 \item[{\rm (ii)}] $(X,\leq)$ has a top element and $Y=\min(X,\leq)$.
\item[{\rm (iii)}] $\alg{A}$ satisfies  \eqref{Eq:nontrivial}, \eqref{Eq:1prime}, and \eqref{Kleenerule}.
\end{itemize}

\end{lemma}
\begin{proof}
(ii) $\Leftrightarrow$ (iii). Immediate from Lemma~\ref{Lem:KleneConds}.

(i) $\Rightarrow$ (ii). By Theorems~\ref{Th:FreeDual} and~\ref{Theo:KleeneDuality}, $\XK(\F_{\KA}(n))\cong\spa{K}^{n}$. The result follows 
straightforwardly from the fact that $(a,\ldots,a)\in K^{n}$ is the top element of $\spa{K}^{n}$, and that $\{0,1\}^{n}=\min(\spa{K}^{n})$.

 (ii) $\Rightarrow$ (i). Let $MC$ be the set of maximal chains of $(X,\leq)$ and let $Z=\{(c,C)\in X\times MC\mid c\in C\}\cup\{x\}$ where $x\notin X\times MC$. 
 Let $\leq_Z$ be the partial order on $Z$ defined as follows:
$$
u\leq_Z v \qquad \Leftrightarrow \qquad v=x\quad \mbox{ or }\quad u=(c_1,C), \ v=(c_2,C), \mbox{ and } c_1\leq c_2.
$$
Clearly, $( Z,\leq_Z)$ is a tree. Defining $Y_Z=\min(Z,\leq_Z)$ and $\sim_Z=\leq_Z\cup\geq_Z$,
it follows from Lemma~\ref{Lem_kleeneSub} that 
$\alg{B}=\AK(Z,  \leq_Z,\sim_Z, Y_Z, \mathcal{P}(Z))
$ belongs to $\cop{IS}(\F_{\KA})$.

Let $\eta\colon Z\to X$ be the map defined by $\eta(c,C)=c$ and let $\eta(x)$ be the top element of $(X ,\leq)$. From the definition of $\leq_Z$,  
$\eta$ is order-preserving, and, since $Y = \min(X,\leq)$, $\eta$ maps $Y_Z$ into $Y$.  It also follows that $\eta$ is onto $X$ and 
$\eta(\sim_Z)=\eta(\leq_Z\cup\geq_Z)\subseteq (\leq\cup\geq)\subseteq \thinspace\sim$. 
Hence $\eta$ is a morphism of Kleene spaces from $(Z,  \leq_Z,\sim_Z,Y_Z,\mathcal{P}(Z))$ onto $(X,\leq,\sim,Y,\tau)$. 
The map $\AK(\eta)\circ e_{\alg{A}}\colon \alg{A}\to \alg{B}$ is therefore a one-to-one homomorphism. Hence 
$\alg{A}\in\cop{IS}(\alg{B})\subseteq \cop{IS}(\F_{\KA})$. 

\end{proof}

So now by Lemmas~\ref{l:mainlemma},~\ref{Lem:KleneConds},~and~\ref{Lem_ISPUkleene}:

\begin{theorem} \label{t:KleeneClauses} 
$\{ \eqref{Eq:nontrivial}, (\ref{Eq:1prime}),  (\ref{Kleenerule})\}$ is a basis for the admissible clauses  of  Kleene algebras.
 \end{theorem}

Turning our attention next to quasi-identities:

\begin{lemma}\label{Lem_ISPkleene}
  Let $\alg{A}$ be a finite Kleene algebra and let $\XK(\alg{A})=(X,\leq,\sim,Y,\tau)$ be its dual space. Then the following are equivalent:
\begin{itemize}
 \item[{\rm (i)}] $\alg{A}\in\cop{ISP}(\F_{\KA})$.
 \item[{\rm (ii)}]   $Y=\min(X,\leq)$. 
\item[{\rm (iii)}] $\alg{A}$ satisfies \eqref{Kleenerule}.
 \end{itemize}
\end{lemma}

\begin{proof}
(ii) $\Leftrightarrow$ (iii). The equivalence follows using  Lemma~\ref{Lem:KleneConds} parts (a) and (c), and the observation that 
${Y=\min(X,\leq)=\emptyset}$ implies $X=\emptyset$.

(i) $\Rightarrow$ (iii). $\KA$ is locally finite, so every finitely generated subalgebra 
of $\F_{\KA}$ is finite and, by Lemma~\ref{Lem_ISPUkleene},  satisfies \eqref{Kleenerule}. 
It follows that $\F_{\KA}$ and hence every algebra in $\cop{ISP}(\F_{\KA})$ satisfies this quasi-identity.

(ii) $\Rightarrow$ (i). Follows by a similar argument to that given in the proof of Lemma~\ref{l:dmaisp} (ii) $\Rightarrow$ (i). 

\end{proof}

Hence, by Lemmas~\ref{l:mainlemma}~and~\ref{Lem_ISPkleene}, we obtain (by different means) the result of~\cite{MR12}:

\begin{theorem}\label{t:BasisKleeneQ}
$\{\eqref{Kleenerule}\}$ is a basis for the admissible quasi-identities of Kleene algebras.
 \end{theorem}

\noindent
Finally, we provide bases for the admissible clauses and admissible quasi-identities of the class of Kleene lattices.

\begin{lemma}\label{l:kleenelattice} 
Let $\alg{A}$  be a finite Kleene lattice. Then the following are equivalent:
\begin{itemize}
\item[{\rm (i)}] $\alg{A} \in \cop{IS}(\F_{\cls{KL}})$.
\item[{\rm (ii)}] $\alg{A}$ satisfies \eqref{Eq:DMnontrivial} and~\eqref{Kleenerule}.
\end{itemize}
\end{lemma}
\begin{proof}
Observe first that $\alg{A}\in\cop{IS}({\F_{\cls{KL}}})$ iff 
 $\overline{\alg{A}}\in\cop{IS}(\F_{\cls{KA}})$. Hence, by Lemma~\ref{Lem_ISPUkleene}, $\alg{A}\in\cop{IS}({\F_{\cls{KL}}})$ iff 
$\overline{\alg{A}}$ satisfies \eqref{Eq:nontrivial}, \eqref{Eq:1prime}, and \eqref{Kleenerule}. Since $\overline{\alg{A}}$  always 
satisfies \eqref{Eq:nontrivial} and $\eqref{Eq:1prime}$, it suffices to prove that $\overline{\alg{A}}$ satisfies~\eqref{Kleenerule} iff 
$\alg{A}$ satisfies~\eqref{Eq:DMnontrivial}  and~\eqref{Kleenerule}. 

Suppose that $\overline{\alg{A}}$ satisfies~\eqref{Kleenerule}. Then clearly $\alg{A}$ satisfies~\eqref{Kleenerule}. 
Also, if $c\in\alg{A}$ satisfies $c=\neg c$, then $c = c\wedge \neg \bot\leq \neg c\vee \bot = c$ in $\overline{\alg{A}}$ 
and, by \eqref{Kleenerule}, it follows that $\top=\neg \bot \leq \bot$, a contradiction. Hence $\alg{A}$ satisfies~\eqref{Eq:DMnontrivial}.

For the converse, suppose that  $\alg{A}$ satisfies~\eqref{Eq:DMnontrivial}  and~\eqref{Kleenerule}. 
Consider $c,d\in\overline{\alg{A}}$ such that $\neg c\leq c$ and $c\wedge \neg d\leq \neg c \vee d$. 
Observe that $d\neq\bot$. Indeed, if $d=\bot$, then $c=c\wedge \top\leq \neg c \vee \bot=\neg c$, 
so $c=\neg c$. But then $c\notin\{\bot,\top\}$ and, as $\alg{A}$ satisfies~\eqref{Eq:DMnontrivial}, also 
$c\notin A$, a contradiction. The following cases remain: 
(i) If $c,d\in A$, then, as $\alg{A}$ satisfies~\eqref{Kleenerule}, also $\neg d\leq d$; 
(ii) If $c\in\{\bot,\top\}$, then, as $\neg c\leq c$, it follows that $c=\top$ and $\neg d=\top\wedge \neg d\leq \bot \vee d=d$; 
(iii) If $d=\top$, clearly $\neg d=\bot\leq \top=d$. 

\end{proof}

\begin{lemma}\label{l:kleenelatticeISP} 
Let $\alg{A}$  be a finite Kleene lattice. Then the following are equivalent:
\begin{itemize}
\item[{\rm (i)}] $\alg{A} \in \cop{ISP}(\F_{\cls{KL}})$.
\item[{\rm (ii)}] $\alg{A}$ satisfies~\eqref{Kleenerule}.
\end{itemize}
\end{lemma}
\begin{proof}
(i) $\Rightarrow$ (ii). 
$\cls{KL}$ is locally finite, so every finitely generated subalgebra of $\F_\cls{KL}$ is finite and, by Lemma~\ref{l:kleenelattice}, 
satisfies \eqref{Kleenerule}. It follows that $\F_\cls{KL}$ and hence every algebra in $\cop{ISP}(\F_\cls{KL})$ satisfies~\eqref{Kleenerule}.

(ii) $\Rightarrow$ (i).  
If $\alg{A}$ satisfies also \eqref{Eq:DMnontrivial}, then by Lemma~\ref{l:kleenelattice}, 
$\alg{A}\in\cop{IS}(\F_{\cls{KL}})\subseteq\cop{ISP}(\F_{\cls{KL}})$. 
If  $\alg{A}$ does not satisfy~\eqref{Eq:DMnontrivial}, there exists $c\in \alg{A}$ such that $\neg c= c$. Then for any $d\in \alg{A}$, 
\[
c\wedge \neg(c\wedge d)\leq c=\neg c\leq \neg c \vee (c \wedge d).
\] 
By~\eqref{Kleenerule}, it follows that $\neg (c \wedge d)\leq c \wedge d$. On the other hand, as $c \wedge d \leq c$, 
\[
c \wedge d \leq c=\neg c\leq \neg(c \wedge d).
\]
It follows that $c \wedge d =\neg(c \wedge d)=c$. So $\alg{A}$ is a trivial algebra and a member of $\cop{ISP}(\F_{\cls{KL}}).$
\end{proof}

Hence, by Lemmas~\ref{l:mainlemma},~\ref{l:kleenelattice}, and~\ref{l:kleenelatticeISP}:

\begin{theorem} \label{t:Kleenemain} 
$\{\eqref{Eq:DMnontrivial},\eqref{Kleenerule}\}$ and $\{\eqref{Kleenerule}\}$ are 
bases for the admissible clauses and admissible quasi-identities of Kleene lattices, respectively.
 \end{theorem}


\subsection*{Acknowledgements}

\noindent
The authors acknowledge support from Marie Curie Reintegration grant PIRG06-GA-2009-256492 and
Swiss National Science Foundation grants 20002{\_}129507 and grant 200021{\_}146748. We thank 
Christoph R{\"o}thlisberger for his helpful comments.


\end{document}